\pdfoutput=1
\documentclass[11pt,a4paper]{article}
\usepackage[margin=27mm]{geometry}
\usepackage[T1]{fontenc}
\usepackage[utf8]{inputenc}
\usepackage{lmodern}
\usepackage{amsmath,amssymb,amsfonts,amsthm}
\usepackage[authoryear,round]{natbib}
\usepackage{algorithm}
\usepackage{algorithmicx}
\usepackage{algpseudocode}
\usepackage{pgfplots}
\usepackage{pgfplotstable}
\pgfplotsset{compat=1.18}
\usepackage{graphicx}
\usepackage{adjustbox}
\usepackage{booktabs}
\usepackage{siunitx}
\usepackage{float}
\usepackage{xcolor}
\usepackage{placeins}
\usepackage[hidelinks]{hyperref}

\theoremstyle{plain}
\newtheorem{theorem}{Theorem}[section]
\newtheorem{lemma}[theorem]{Lemma}
\newtheorem{proposition}[theorem]{Proposition}
\theoremstyle{remark}
\newtheorem{remark}[theorem]{Remark}
\newtheorem{example}[theorem]{Example}
\newtheorem{assumption}[theorem]{Assumption}
\newtheorem{definition}[theorem]{Definition}

\title{Low-order CR--RT equilibrated-flux certification for semilinear problems on anisotropic meshes}
\author{Hiroki Ishizaka\\
  \small Team FEM, Matsuyama, Japan\\
  \small \texttt{h.ishizaka005@gmail.com}\\
  \small \url{https://teamfem.github.io/hiroki_ishizaka/}\\
  \small ORCID: 0000-0002-5892-7488}
\date{Accepted for publication in \emph{Journal of Computational and Applied Mathematics} (2026)}

\begin{document}
\maketitle

\begin{abstract}
We study a posteriori certification of semilinear diffusion--reaction problems on anisotropic simplicial meshes. Given a conforming finite element approximation, we bound its dual-norm residual by solving an auxiliary Crouzeix--Raviart problem and reconstructing an equilibrated Raviart--Thomas flux. Together with a coercivity estimate and a local derivative-Lipschitz bound, this gives a Newton--Kantorovich test for existence and local uniqueness in a computable energy-norm ball. We then use the verified radius to bound volume-type quantities of interest. The adjoint estimate accounts for the primal and adjoint residuals and retains explicit Taylor remainders for the nonlinear equation and, where relevant, the output functional. We also prove a conditional flux-mismatch consistency result on anisotropic mesh families. The CR--RT reconstruction and adjoint identities are established tools; here they provide a concrete low-order route to nonlinear solution and output certification. Numerical tests illustrate the bounds, but their double-precision evaluation does not constitute an interval-verified computation.
\end{abstract}


\section{Introduction}\label{sec=intro}

We consider a semilinear elliptic problem $\mathcal F(u)=0$ in $V^*$, where $V:=H_0^1(\Omega)$. Given a conforming finite element approximation $\tilde u_h\in V$, we seek a computable radius $\rho>0$ such that an exact solution lies in
\begin{align*}
B_\rho
:=\{w\in V:\ \|w-\tilde u_h\|_V\leq\rho\}.
\end{align*}

We obtain the residual bound for a Newton--Kantorovich existence test from a low-order CR--RT flux reconstruction. The verified radius is then used to derive bounds for volume quantities of interest, with explicit remainders for both the semilinear equation and the output functional. For anisotropic meshes, we also establish a conditional flux-mismatch consistency result in the Poisson setting.

Equilibrated fluxes and hypercircle estimates are standard tools in a posteriori analysis \citep{BraPilSch09,ErnVoh15}. Newton--Kantorovich (NK) arguments have also been used extensively in computer-assisted existence proofs \citep{Nakao1988,Plum2009,NakPluWat19,KinoshitaKimuraNakao2013}. In particular, \citet{TakayasuLiuOishi2013} used finite element residual bounds and hypercircle estimates in an NK-type test for semilinear elliptic equations. The CR--RT connection, including its use on anisotropic meshes, is also known \citep{Mar85,IshKobTsu21b}. We use these results rather than propose a new flux identity or a new existence principle.

Guaranteed bounds for quantities of interest are similarly well established. \citet{AinsworthRankin2012QoI} obtained computable output bounds for linear second-order elliptic problems. \citet{MallikVohralikYousef2020QoI} used primal and dual potential and flux reconstructions for linear elliptic problems, including inexact algebraic solvers. \citet{ParesNguyenDiezPeraire2021QoI} derived guaranteed output bounds for the Poisson problem from primal and adjoint reconstructions, including low-order equilibrated fluxes. The adjoint error representation and the use of two residual reconstructions are therefore not new.

What requires attention in the semilinear setting is the remainder of the equation as well as that of the output. For a conforming approximation $\tilde u_h$, our CR--RT flux bounds the residual by $\mathfrak r$ in (C1). A coercivity bound $\alpha>0$ and a local derivative-Lipschitz bound $L(\rho)$ provide (C2) and (C3). The conditions $p(\rho)\leq0$ and $q(\rho)<1$ in Theorem~\ref{thm:NK} then establish $\|u-\tilde u_h\|_V\leq\rho$. This estimate is carried into Theorem~\ref{thm:qoi:adjoint}. Alongside the primal residual pairing and an adjoint residual bound, its error budget contains terms in $L(\rho)\rho^2$ and, for a nonlinear output, $M_{\mathcal J}(\rho)\rho^2$. These terms bound the two Taylor remainders on the verified ball. 

The residual bound is obtained from an auxiliary scalar CR solve followed by a cellwise Marini-type $\mathbb{RT}^0$ reconstruction. We apply the same procedure to the linear adjoint residual in Section~\ref{subsec:qoi:adjoint}. This avoids a patchwise mixed saddle-point solve, but still requires auxiliary CR problems; we make no general claim about runtime or complexity. Proposition~\ref{prop:anisotropic-mismatch-consistency} also has a limited scope. It gives flux-mismatch consistency on anisotropic mesh families when the CR and conforming approximations are energy-consistent, rather than an aspect-ratio-uniform efficiency estimate; see \citep{IshKobTsu21b,IshizakaKobayashiTsuchiya2023}.

For the stability test, we mainly use symmetric coercive linearizations. Other linearizations would require a separate verification of invertibility. The bounds proved below are analytical results under the stated assumptions. The numerical examples are evaluated in ordinary double precision, without rounding-error enclosures, and should not be interpreted as computer-assisted interval proofs.

Section~2 sets out the problem and discretization. Section~3 states the NK test, Sections~4--7 develop its computable inputs, and Section~8 gives the output bounds. Section~9 contains the numerical tests.

\section{Problem setting, discretization, and verification neighborhoods}
\label{sec:problem}

\subsection{Preliminaries}
Let $\Omega \subset \mathbb{R}^d$, $d\in\{2,3\}$ be a bounded polyhedral Lipschitz domain. We set
\begin{align*}
\displaystyle
V := H^1_0(\Omega), \quad V^* := H^{-1}(\Omega),
\end{align*}
and write $\langle \cdot,\cdot\rangle$ for the duality pairing between $V^*$ and $V$. For a bounded linear operator $K\in\mathcal{L}(V,V^*)$, we use the induced operator norm
\begingroup
\begin{align}\label{eq:op-norm}
\|K\|_{\mathcal{L}(V,V^*)}
:= \sup_{\delta u \in V\setminus\{0\}} \frac{\|K\delta u\|_{V^*}}{\|\delta u\|_V}
= \sup_{\delta u,v\in V\setminus\{0\}}
\frac{\langle K\delta u,v\rangle}{\|\delta u\|_V\,\|v\|_V}.
\end{align}
\endgroup
The Fr\'echet derivative $D \mathcal F(w)\in\mathcal{L}(V,V^*)$ is characterized by
\begin{align*}
\langle D \mathcal F (w) \delta u, v\rangle
:=\left.\frac{d}{d\varepsilon}\right|_{\varepsilon=0}\langle \mathcal F (w+\varepsilon\delta u),v\rangle,
\quad \delta u,v\in V.
\end{align*}

\subsection{Semilinear model} \label{sec:semilinear-problem}
We study semilinear diffusion--reaction problems in variational form. Let $f \in V^*$. Given a fixed diffusion tensor $A:\Omega\to\mathbb{R}^{d\times d}$, we seek $u\in V$ such that
\begin{align} \label{eq:semi-weak}
\int_\Omega A(x) \nabla u \cdot \nabla v dx + \int_\Omega c(x,u) v dx = \langle f,v\rangle \quad \forall v\in V,
\end{align}
with a nonlinear reaction $c(x,u) := b(x,u)  u$, where $b:\Omega\times\mathbb{R}\to\mathbb{R}$. Here, $A$ is symmetric and uniformly positive definite (but may vary with $x$), and it does not depend on $u$, and  $b$ may depend on $x$ and $u$ (e.g.,  $b(x,u)= u^2$).

\begin{remark}
The weak formulation permits $f\in V^*$, but the residual estimator in Sections~\ref{sec:flux}--\ref{sec:fluxrecon} uses $L^2$ data. We assume $f\in L^2(\Omega)$ and, for the computed state, $c(\cdot,\tilde u_h)\in L^2(\Omega)$. Thus, the known source $s_h:=c(\cdot,\tilde u_h)-f$ and the element residuals belong to $L^2(\Omega)$. This ensures that the cellwise projections and oscillation terms in the equilibrated-flux bound are well-defined. Rougher data would require a different residual treatment.
\end{remark}

\begin{remark}

In this paper, we focus exclusively on the semilinear diffusion-reaction framework. Specifically, we do not address quasilinear diffusion operators in this study. Extending the certification workflow to encompass quasilinear diffusion would necessitate further control over the linearized diffusion operator and the associated Lipschitz bounds. Consequently, we consider this a distinct issue and defer it to future research.

\end{remark}

\subsection{Operator form and linearization}\label{subsec:operator-linearization}
We write the semilinear model as an operator equation $\mathcal F(u)=0$ in $V^*$ and consider the Fr\'echet derivative. All differentiability and Lipschitz assumptions are posed on an open set $\mathcal U\subset V$. In particular, $\mathcal U$ contains the computed state $\tilde u_h$, and $\mathcal F$ is Fr\'echet differentiable there; see Sections~\ref{subsec:U-and-balls} and~\ref{subsec:NK-assumptions}. Furthermore, we use the Sobolev embedding $V\hookrightarrow L^p(\Omega)$ with $p\in[2,\infty)$ if $d=2$ and $p\in[2,6]$ if $d=3$. We denote the H\"older conjugate by $p^{\prime} :=\frac{p}{p-1}$.

In the semilinear case, we define $\mathcal F:V\to V^*$ as
\begin{align}\label{eq:F-semi}
\langle \mathcal F(w),v\rangle
:= \int_\Omega A(x)\nabla w\cdot\nabla v dx + \int_\Omega c(x,w)  \, v dx - \langle f,v\rangle
\quad \forall v\in V.
\end{align}
We assume that $c(x,\cdot)$ is continuously differentiable for almost every $x$ and write $\partial_s c(x,s)$ for its derivative with respect to $s$. Then, the Fr\'echet derivative $D \mathcal F(w)\in\mathcal{L}(V,V^*)$ satisfies
\begin{align}\label{eq:DF-semi}
\langle D \mathcal F(w)\,\delta u, v\rangle
= \int_\Omega A(x)\nabla(\delta u)\cdot\nabla v dx
 + \int_\Omega \partial_s c(x,w)  \, \delta u \,  v dx
\quad \forall  \delta u,v\in V,
\end{align}
where $\partial_s c(x,w) = b(x,w) + w \partial_s b(x,w)$. Here, $s$ denotes the scalar state argument. On the open set $\mathcal U$, we assume: for all $w\in \mathcal U$ one has $c(\cdot,w)\in L^{p^{\prime}}(\Omega)$ and $\displaystyle \partial_s c(\cdot,w)\in L^{\frac{p}{p-2}}(\Omega)$ (with the usual convention $ \frac{p}{p-2}=\infty$ if $p=2$).

\subsection{$A$-weighted energy norm}\label{subsec:energy-norm}
Verification will be formulated in terms of an energy norm on $V$ and its dual norm on $V^*$. We equip $V$ with a fixed energy norm induced by a uniformly elliptic matrix field $A \in L^\infty(\Omega)^{d\times d}$. We assume that $A$ is symmetric and there exist constants $0<\alpha_0 \leq \beta_0<\infty$ such that
\begin{align*}
\displaystyle
\alpha_0|\xi|^2 \leq A(x)\xi\cdot\xi \leq \beta_0|\xi|^2 \quad \text{for a.e. }x \in \Omega, \quad \forall \xi \in \mathbb{R}^d.
\end{align*}
We define
\begin{align}
\displaystyle
\|v\|_V := \left(\int_\Omega A \nabla v \cdot \nabla v dx \right)^{\frac{1}{2}}, \quad v\in V, \label{eq:V-norm}
\end{align}
and the associated dual norm on $V^*$ as
\begin{align}
\displaystyle
\| \ell \|_{V^*}:=\sup_{v\in V\setminus\{0\}}\frac{\langle \ell,v\rangle}{\|v\|_V}. \label{eq:dual-norm}
\end{align}

\subsection{Finite elements and the computed state}\label{subsec:discretization}
Let $\mathbb{T}_h = \{ T \}$ be a simplicial mesh of $\overline{\Omega}$ made up of closed $d$-simplices, such as $\displaystyle \overline{\Omega} = \bigcup_{T \in \mathbb{T}_h} T$, with $h := \max_{T \in \mathbb{T}_h} h_{T}$, where $ h_{T} := \mathrm{diam}(T)$. For simplicity, we assume that $\mathbb{T}_h$ is conforming: that is, $\mathbb{T}_h$ is a simplicial mesh of $\overline{\Omega}$ without hanging nodes. Let $|\cdot|_d$ denote the $d$-dimensional Hausdorff measure.

Let $V_h \subset V$ be a conforming finite element space consisting of continuous, piecewise polynomial functions of degree $k \geq 1$ with vanishing trace on $\partial\Omega$. The discrete problem reads: find $u_h\in V_h$ such that
\begin{align}
\displaystyle
\langle \mathcal F(u_h),v_h\rangle=0 \quad \forall v_h\in V_h. \label{eq:disc_weak}
\end{align}
In practice, we obtain a computed approximation $\tilde u_h\in V_h$ (e.g., a terminated Newton iterate). The certification layer developed below takes $\tilde u_h$ as input and provides guarantees for the continuous solution $u\in V$.

\begin{remark}
The verification is formulated for a conforming state $\tilde u_h\in V$. If the primal approximation is nonconforming, as in the CR or interior penalty discontinuous Galerkin methods, a conforming companion $\tilde u_h^c\in V$ may first be constructed by averaging or a suitable companion operator. The residual estimator and NK test can then be applied to this companion. An error estimate measured against the original nonconforming approximation would also require control of the nonconforming part; see \cite{CarstensenMerdon2013,ErnVoh15}. Such an extension is not analyzed here.
\end{remark}

\subsection{Open neighborhood and candidate balls}\label{subsec:U-and-balls}
We fix an open neighborhood $\mathcal U\subset V$ of the computed state $\tilde u_h$ on which $\mathcal F$ is Fr\'echet differentiable. For candidate radius $\rho>0$, we consider closed balls
$B(\tilde u_h, \rho) =\{w\in V:\ \|w-\tilde u_h\|_V \leq \rho\}\subset \mathcal U$, and determine an admissible $\rho$ a posteriori by checking scalar verification inequalities; see Section~\ref{sec:NK}.

\subsection{Quantities of interest}\label{subsec:qoi}
In many applications, one is not primarily interested in the full field $u$, but in a specific output functional (quantity of interest, QoI) $\mathcal J(u)$. Our goal is to provide a posteriori-verified bound for $\mathcal J(u)$ once the existence (and local uniqueness) of $u$ has been certified in a neighborhood of the computed state $\tilde u_h$.

Let $\mathcal J:V\to\mathbb{R}$ be well-defined on the open verification neighborhood $\mathcal U\subset V$ introduced in Sections~\ref{subsec:U-and-balls}--\ref {subsec:NK-assumptions} (or on an open set containing the verified ball). We assume that $\mathcal J$ is Fr\'echet differentiable on $\mathcal U$ and denote its derivative by $D \mathcal J (w)\in V^*$. In the output certification step, we will bound the variation of $D \mathcal J$ on candidate balls by a computable function $L_{\mathcal J}(\rho)$, in complete analogy with the bound for $D \mathcal F$.

\paragraph*{\textbf {Typical examples.}}
This study confines the certified output analysis to volume-type quantities of interest (QoIs) for which the derivatives possess an $L^2(\Omega)$-density. Typical examples include:
\begin{itemize}
\item {Linear volume outputs:} $\mathcal J(u)=\int_\Omega \psi\,u\,dx$ with $\psi\in L^2(\Omega)$. Then, $D \mathcal J(w)(\delta u)=\int_\Omega \psi\,\delta u\,dx$.
\item {Local averages:} $\mathcal J(u)=|\omega|_d^{-1}\int_\omega u\,dx$ for a measurable subdomain $\omega\subset\Omega$, provided the corresponding density belongs to $L^2(\Omega)$ or is replaced by a suitable $L^2$ regularization.
\item {Nonlinear volume outputs:} $\mathcal J(u)=\int_\Omega \Phi(x,u)\,dx$, where $\Phi(x,\cdot)$ is differentiable and the derivative $\partial_s\Phi(x,w)$ defines an $L^2(\Omega)$-density on the verification neighborhood. Then, $D \mathcal J(w)(\delta u)=\int_\Omega \partial_s\Phi(x,w)\,\delta u\,dx$.
\end{itemize}

\begin{remark}
Boundary outputs, normal-flux outputs, and other trace-type functionals are not addressed in this paper. A rigorous analysis of these would necessitate further trace estimates, weak boundary representations, or flux-specific residual arguments, which are deferred to future research.
\end{remark}

\paragraph*{\textbf{Adjoint viewpoint.}}
For a fixed linearization point $w\in\mathcal U$, the output error is represented through an adjoint problem driven by the functional $D \mathcal J(w)\in V^*$. This connection underpins the goal-oriented bounds developed later, and it motivates formulating assumptions on $\mathcal J$ directly in terms of its derivative.

\section{NK certification: existence and localization}\label{sec:NK}
This section provides an abstract verification step for the existence (and local uniqueness) of a weak solution of $\mathcal F(u)=0$ near a computed approximation $\tilde u_h\in V_h$. The verification radius $\rho$ is not known a priori; it is determined a posteriori from computable scalar inequalities. The results of this section are classical in nonlinear functional analysis and Newton theory; see, for instance, \cite{OrtRhe70, Deim85,Deuf11}.

\subsection{Local assumptions and scalar ingredients}\label{subsec:NK-assumptions}
Let $\mathcal U \subset V$ be open such that $\mathcal F:V\to V^*$ is Fr\'echet differentiable on $\mathcal U$ and $\tilde u_h\in \mathcal U$. Fix $\bar\rho>0$ such that $B(\tilde u_h,\bar\rho)\subset \mathcal U$. For any $0<\rho\le \bar\rho$, we consider
\begin{align*}
\displaystyle
B_\rho =\{w\in V:\ \|w-\tilde u_h\|_V\le \rho\}\subset \mathcal U,
\end{align*}
where $\|\cdot\|_V$ is the energy norm from \eqref{eq:V-norm}. We write $\|\cdot\|_{\mathcal L(V,V^*)}$ for the induced operator norm \eqref{eq:op-norm}. We linearize at the computed state and for any $w \in \mathcal U$, we set
\begin{align*}
\displaystyle
\mathcal L_w := D \mathcal F(w)\in\mathcal L(V,V^*).
\end{align*}

We assume that the following quantities are available (as \emph{certified upper bounds}):
\begin{enumerate}
\item[(C1)] {Residual bound at $\tilde u_h$.}
There exists a computable number $\mathfrak r\ge 0$ such that
\begin{align}\label{eq:C1}
\| \mathcal F(\tilde u_h)\|_{V^*}\leq \mathfrak r.
\end{align}
\item[(C2)] {Stability/invertibility of the linearization.}
The operator $\mathcal L_{\tilde u_h}:V\to V^*$ is an isomorphism and there exists a constant $\alpha>0$ such that
\begin{align}\label{eq:C2}
\|\mathcal L_{\tilde u_h}^{-1}\|_{\mathcal L(V^*,V)} \le \alpha^{-1}.
\end{align}

For later use in Theorem~\ref{thm:NK}, we introduce the scaled residual parameter

\begin{align}\label{eq:eta-def}
\eta := \frac{\mathfrak r}{\alpha}.
\end{align}

The residual, normalized by the stability lower bound, is employed in the admissibility function $p(\rho)$.

\item[(C3)] {Local Lipschitz bound for $D \mathcal F$ on candidate balls.}
There exists a (computable) nondecreasing function $L:(0,\bar\rho]\to(0,\infty)$ such that for every $0<\rho\le \bar\rho$ and all $w,z\in B_\rho$,
\begin{align}\label{eq:C3}
\| \mathcal L_{w} - \mathcal L_{z} \|_{\mathcal L(V,V^*)}\le L(\rho)\,\|w-z\|_{V}.
\end{align}
\end{enumerate}

Assumptions \textup{(C1)}--\textup{(C3)} will be realized by the constructions in the subsequent sections: \textup{(C1)} from equilibrated-flux residual bounds, \textup{(C2)} from a certified stability estimate for the linearized operator, and \textup{(C3)} from model-dependent coefficient-derivative bounds.

\begin{remark}[How (C2) is verified in this paper]\label{rem:C2-symmetric}
In nonsymmetric settings, one would typically verify (C2) via an inf--sup theory (possibly involving
the adjoint operator). In the present work, we focus on the symmetric elliptic case, for which a
coercivity estimate is sufficient.
More precisely, we compute $\alpha>0$ such that
\begin{align*}
\displaystyle
\langle \mathcal L_{\tilde u_h} v, v\rangle \geq \alpha \|v\|_V^2 \quad \forall v\in V.
\end{align*}
Then, $\mathcal L_{\tilde u_h}:V\to V^*$ is bijective and $\|\mathcal L_{\tilde u_h}^{-1}\|_{\mathcal L(V^*,V)}\leq \alpha^{-1}$ by the Lax--Milgram lemma; hence (C2) holds.

This paper intentionally imposes a restriction. When the linearized reaction term includes a negative component, the bilinear form may lack coercivity or approach singularity, even if the negative component is minimal, as indicated by a basic Poincar\'e-type estimate. These non-monotone conditions are significant in verified computation and necessitate more refined techniques for verifying invertibility, such as Nakao-type methods and fixed-point approaches for semilinear elliptic boundary value problems (see, for example, \cite{NakPluWat19,Plum2009}). This study does not pursue verification of non-coercive stability.

\end{remark}

\subsection{Existence and local uniqueness via a contraction}\label{subsec:NK-theorem}

\begin{lemma}[Quadratic remainder bound]\label{lem:remainder}
Assume \textup{(C3)} and fix $0<\rho\le \bar\rho$. Then, for any $w\in B_\rho$ one has
\begin{align}\label{eq:remainder}
\|\mathcal F(w) - \mathcal F(\tilde u_h)-\mathcal L_{\tilde u_h}(w-\tilde u_h)\|_{V^*}
\leq \frac{L(\rho)}{2}\,\|w-\tilde u_h\|_{V}^{2}.
\end{align}
\end{lemma}

\begin{proof}
A proof can be found in Appendix	\ref{Appendix=A=1}.
\end{proof}

We define the Newton (Kantorovich) map $\mathcal N:\mathcal U\to V$ by
\begin{equation}\label{eq:Newton-map}
\mathcal N(w) := w - \mathcal L_{\tilde u_h}^{-1} \mathcal F(w),
\end{equation}
where $\mathcal L_{\tilde u_h}^{-1} \mathcal F(w)\in V$ is the unique solution $\delta\in V$ of the linear problem
\begin{equation}\label{eq:Newton-correction}
\mathcal L_{\tilde u_h} \delta = \mathcal F(w)\quad\text{in }V^*.
\end{equation}
Equivalently, $\mathcal N(w)=w-\delta$ with $\delta$ given by \eqref{eq:Newton-correction}. A fixed point $u=\mathcal N(u)$ is therefore equivalent to $\mathcal F(u)=0$.

Under \textup{(C1)}--\textup{(C3)}, we show that $\mathcal{N}$ maps $B_{\rho}$ into itself and is a contraction on $B_{\rho}$.

The subsequent localization result represents a standard NK-type argument, which is closely associated with fixed-point and radii-polynomial formulations employed in verified computation. This inclusion in the current notation serves to explicitly demonstrate the integration of the three computable quantities (C1)--(C3) into the finite element certification procedure; refer to \cite{Plum2009,NakPluWat19} for further details.

\begin{theorem}[Verified existence and localization]\label{thm:NK}
Assume \textup{(C1)}--\textup{(C3)}. Fix $0<\rho \leq \bar\rho$ and we define
\begin{align}\label{eq:pq}
p(\rho):=\eta + \frac{L(\rho)}{2\alpha}\rho^{2}-\rho,
\quad
q(\rho):=\frac{L(\rho)}{\alpha}\rho,
\end{align}
where $\eta$ is given by \eqref{eq:eta-def}. If
\begin{align}\label{eq:NK-conds}
p(\rho)\leq 0
\quad \text{and} \quad
q(\rho) < 1,
\end{align}
then there exists a unique $u\in B_\rho$ such that $\mathcal F(u)=0$. Furthermore, $\mathcal N$ maps $B_\rho$ into itself and is a contraction on $B_\rho$ with contraction factor $q(\rho)$.
\end{theorem}

\begin{proof}
Step 1: $\mathcal N(B_\rho)\subset B_\rho$. Let $w\in B_\rho$. From the definition of $\mathcal{N}$,
\begin{align*}
\displaystyle
\mathcal{N}(w)-\tilde u_h
&= (w -\tilde u_h ) - \mathcal L_{\tilde u_h}^{-1} \mathcal{F}(w) \\
&= \mathcal L_{\tilde u_h}^{-1} \mathcal L_{\tilde u_h} (w -\tilde u_h ) - \mathcal L_{\tilde u_h}^{-1} \mathcal{F}(w) \\
&=  \mathcal L_{\tilde u_h}^{-1}\left( \mathcal{F}(\tilde u_h) - \mathcal{F}(w) + \mathcal L_{\tilde u_h} (w -\tilde u_h ) - \mathcal{F}(\tilde u_h)  \right).
\end{align*}
Using \eqref{eq:C1}, \eqref{eq:C2}, and Lemma~\ref{lem:remainder}, we obtain
\begin{align*}
\displaystyle
\|\mathcal N(w)-\tilde u_h\|_{V}
&\leq \|\mathcal L_{\tilde u_h}^{-1}\mathcal F(\tilde u_h)\|_{V}
 + \|\mathcal L_{\tilde u_h}^{-1}\|_{\mathcal L(V^*,V)}\,
   \|\mathcal F(w)- \mathcal F(\tilde u_h)-\mathcal L_{\tilde u_h}(w-\tilde u_h)\|_{V^*} \\
&\leq \frac{\|\mathcal F(\tilde u_h)\|_{V^*}}{\alpha} + \frac{1}{\alpha}\cdot \frac{L(\rho)}{2}\,\rho^2
\le \eta + \frac{L(\rho)}{2\alpha}\rho^2.
\end{align*}
Thus, $p(\rho) \leq 0$ implies $\|\mathcal N(w)-\tilde u_h\|_V \leq \rho$, i.e., $\mathcal N(w)\in B_\rho$.\\

\noindent
Step 2: $\mathcal{N}$ is a contraction on the ball. 
Let $w,z\in B_\rho$ and set $\delta:=w-z$. For $t\in[0,1]$, we have
\begin{align*}
\displaystyle
z + t\delta=(1-t) z+ t w \in B_\rho
\end{align*}
since $B_\rho$ is convex. By the fundamental theorem of calculus,
\begin{align*}
\displaystyle
\mathcal F(w)- \mathcal F(z) = \int_0^1 \mathcal L_{z + t \delta}\delta dt,
\end{align*}
which leads to
\begin{align*}
\displaystyle
\mathcal F(w)-\mathcal F(z)-\mathcal L_{\tilde u_h}\delta=\int_0^1 \left ( \mathcal L_{z + t \delta}  - \mathcal L_{\tilde u_h} \right)\delta dt.
\end{align*}
Because $z+t\delta\in B_\rho$ and $\tilde u_h\in B_\rho$, \eqref{eq:C3} yields
\begin{align*}
\displaystyle
\| \mathcal L_{z + t \delta} - \mathcal L_{\tilde u_h} \|_{\mathcal L(V,V^*)} \leq L(\rho) \|z+t\delta-\tilde u_h\|_V \leq L(\rho)\,\rho.
\end{align*}
Therefore,
\begin{align*}
\displaystyle
\| \mathcal F(w)- \mathcal F(z)-\mathcal L_{\tilde u_h}(w-z)\|_{V^*}
&\leq \int_0^1 \| \mathcal L_{z + t \delta} - \mathcal L_{\tilde u_h} \|_{\mathcal L(V,V^*)} \|\delta\|_V dt \\
&\leq \int_0^1 L(\rho) \rho \|\delta\|_V dt
= L(\rho) \rho \|w-z\|_V.
\end{align*}
Using \eqref{eq:C2}, we conclude
\begin{align*}
\displaystyle
\|\mathcal N(w)-\mathcal N(z)\|_{V}
&= \|\mathcal L_{\tilde u_h}^{-1}\left ( \mathcal F(w)- \mathcal F(z)-\mathcal L_{\tilde u_h}(w-z) \right)\|_V \\
&\leq \frac{L(\rho)}{\alpha}\rho\,\|w-z\|_V
= q(\rho)\,\|w-z\|_V.
\end{align*}
Thus, $q(\rho)<1$ implies that $\mathcal N$ is a contraction on $B_\rho$.

The Banach fixed-point theorem yields a unique fixed point $u\in B_\rho$ of $\mathcal N$,
and $\mathcal N(u)=u$ is equivalent to $\mathcal F(u)=0$.
\end{proof}

\begin{remark}[Practical verification workflow]
The practical verification workflow is as follows.
\begin{description}
  \item[(1)] Compute a discrete approximation $\tilde u_h$ using a standard numerical solver.
  \item[(2)] Compute {certified} bounds for the three inputs in \textup{(C1)--(C3)}:
  a dual-norm residual bound $\mathfrak r$ at $\tilde u_h$, a stability bound
  $\|\mathcal L_{\tilde u_h}^{-1}\|_{\mathcal L(V^*,V)} \leq \alpha^{-1}$ for $\mathcal L_{\tilde u_h}=D \mathcal F(\tilde u_h)$,   and a Lipschitz-type bound $L(\rho)$ valid on $B_\rho$.
  \item[(3)] Choose a candidate radius $\rho$ and evaluate the scalar conditions in
  Theorem~\ref{thm:NK}. If $p(\rho)\le 0$ and $q(\rho)<1$, then a solution exists and is
  locally unique in $B_\rho$.
\end{description}

\smallskip
The analytical verification reduces to bounds for the scalar quantities above. In a fully computer-assisted implementation, rounding errors in the discrete solves and in the evaluations of $\mathfrak r$, $\alpha$, and $L(\rho)$ would also have to be controlled rigorously; outward rounding or interval arithmetic can then be used for the scalar checks of $p(\rho)$ and $q(\rho)$. We do not carry out this implementation here.
\end{remark}

\begin{remark}[Practical choice of the verification radius $\rho$]\label{rem:rho-choice}
A candidate $0<\rho\leq\bar\rho$ is accepted only when $p(\rho)\leq0$ and $q(\rho)<1$. The choice $\rho=2\eta$ is checked first, when it lies in the permitted range. For the affine Lipschitz bound $L(\rho)=L_0+L_1\rho$, $L_0\geq0$, $L_1>0$, a failed initial test is followed by a test at the minimum of $p$ on $(0,\bar\rho]$. This avoids missing an admissible interval between two trial radii. Once an admissible radius has been found, the bisection in Algorithm~\ref{alg:rho} can be used to approach the upper endpoint of the admissible interval, as justified by Proposition~\ref{prop:alg-rho-correct}. A separate bisection towards the lower endpoint may yield a smaller admissible radius and a sharper bound for the quantity of interest. For more general bounds $L(\rho)$, admissibility must be checked independently. We do not claim that dyadic shrinking is exhaustive.
\end{remark}

\section{Certifying (C1): a guaranteed residual estimate via equilibrated fluxes}\label{sec:flux}

This section provides the computable residual bound required for the first verification condition (C1). We first derive a residual representation in terms of an arbitrary $H(\mathrm{div})$-conforming flux and then specialize it to an equilibrated-flux estimate. This separates the abstract NK verification step from the finite element construction of the certificate flux.

\subsection{Residual representation}\label{subsec:flux_repr}
Let $\tilde u_h\in V_h$ be the computed approximation. We recall the residual functional $\mathcal F(\tilde u_h)\in V^*$ defined in Section \ref{subsec:operator-linearization}. Let $\sigma_h \in H(\mathrm{div};\Omega)$ be any $H(\mathrm{div})$-conforming flux. Using elementwise integration by parts, we have for all $v\in V$,
\begin{align}
\langle \mathcal F(\tilde u_h),v\rangle
&=\int_\Omega \left( A \nabla \tilde u_h - \sigma_h \right)\cdot\nabla v\,dx
 + \sum_{T\in\mathbb T_h}\int_T \mathcal R_T(\sigma_h)\, v\,dx,
\label{eq:res_split}
\end{align}
where the {element residual} $\mathcal R_T(\sigma_h)$ is given by
\begin{align} \label{eq:elem_residual}
\displaystyle
\mathcal R_T(\sigma_h):= c(x,\tilde u_h) - f - \nabla\!\cdot \sigma_h \quad \text{(semilinear diffusion--reaction)}.
\end{align}

As stated in Section \ref{subsec:energy-norm}, we consider the computable tensor field $A \in L^\infty(\Omega;\mathbb R^{d\times d})$ that is symmetric and uniformly positive definite: there exist constants $0<\alpha_0\le \beta_0<\infty$ such that
\begin{align*}
\displaystyle
\alpha_0|\xi|^2 \leq A (x)\xi\cdot\xi \leq \beta_0|\xi|^2
\quad \text{for a.e.\ }x\in\Omega,\ \forall\,\xi\in\mathbb R^d.
\end{align*}
Recall that this defines the energy norm
\begin{align*}
\displaystyle
\|v\|_V = \left( \int_\Omega A \nabla v\cdot\nabla v\,dx \right)^{\frac{1}{2}},
\end{align*}
and implies $A^{-1}\in L^\infty(\Omega;\mathbb R^{d\times d})$, so that
\begin{align*}
\displaystyle
\|w\|_{A^{-1}} = \left( \int_\Omega w\cdot A^{-1}w dx \right)^{\frac{1}{2}}
\end{align*}
is well defined for all $w\in L^2(\Omega)^d$. For a vector field $v\in L^2(\Omega)^d$, we also write
$\|v\|_{A}:=\bigl(\int_\Omega v\cdot Av\,dx\bigr)^{1/2}$.

\begin{lemma}[Weighted Cauchy--Schwarz]\label{lem:wCS}
Let $A \in L^\infty(\Omega;\mathbb R^{d\times d})$ be symmetric and uniformly positive definite.
Then, for all $u,v\in L^2(\Omega)^d$, one has
\begin{align}
\displaystyle
\left |\int_\Omega u\cdot v\,dx \right |
\leq
\left (\int_\Omega u\cdot A^{-1}u\,dx \right)^{\frac{1}{2}}
\left (\int_\Omega v\cdot A v\,dx\right)^{\frac{1}{2}} = \| u\|_{A^{-1}} \|v\|_A. \label{eq:wCS}
\end{align}
\end{lemma}

\begin{proof}
A proof can be found in Appendix	\ref{Appendix=A=2}.
\end{proof}

\begin{remark}
Any reconstructed flux $\sigma_h$ used in the certification layer is assumed to satisfy $\sigma_h\in H(\mathrm{div};\Omega)\subset L^2(\Omega)^d$. Consequently, the flux mismatch $w :=\sigma_h-A \nabla \tilde u_h$ belongs to $L^2(\Omega)^d$, and the estimate
\begin{align*}
\displaystyle
\left |\int_\Omega (\sigma_h - A \nabla \tilde u_h ) \cdot \nabla v dx \right|
\leq \| \sigma_h - A \nabla \tilde u_h \|_{A^{-1}}\ \|v\|_V
\quad \forall v\in V
\end{align*}
is meaningful.
\end{remark}

\subsection{Equilibration and oscillation}\label{subsec:equil-osc}
The second term in \eqref{eq:res_split} involves the element residual $\mathcal R_T(\sigma_h)$. To control this contribution in the dual norm induced by $\|\cdot\|_V$ without introducing global domain constants, we eliminate the elementwise constant mode of $v$ by enforcing an elementwise mean condition.

We require the following equilibration of means:
\begin{align}
\int_T \mathcal R_T(\sigma_h)\,dx = 0 \quad \forall\,T\in\mathbb T_h. \label{eq:mean-equil}
\end{align}
The purpose of \eqref{eq:mean-equil} is to bound the element-residual term in \eqref{eq:res_split} using only local estimates, i.e., without global Poincar\'e-type constants. Condition \eqref{eq:mean-equil} is standard in equilibrated-residual constructions and can be enforced within RT-type flux spaces; see, e.g., \cite{BraPilSch09,ErnVoh15}. 

Under \eqref{eq:mean-equil}, for each element $T$ we may subtract the element average
\begin{align*}
\displaystyle
\bar v_T := \frac{1}{|T|_d}\int_T v\,dx
\end{align*}
to obtain
\begin{align}
\displaystyle
\int_T \mathcal R_T(\sigma_h)\,v\,dx
= \int_T \mathcal R_T(\sigma_h)\,(v-\bar v_T)\,dx.
\end{align}
We then use the Poincar\'e inequality on a simplex (see, e.g., \cite{PayWei60}):
\begin{align}
\|v-\bar v_T\|_{L^2(T)} \le \frac{h_T}{\pi}\,\|\nabla v\|_{L^2(T)}. \label{eq:local-poincare}
\end{align}
Importantly, \eqref{eq:local-poincare} depends only on $h_T$ and does not invoke angle conditions.

Sharper bounds for Poincar\'e constants on simplices are available in the literature; see, for example, \cite{LaugesenSiudeja2010}. In the computations in Section~\ref{sec:numerics}, however, we keep the explicit bound used here in order to maintain a simple and reproducible certification procedure.

\subsection{A guaranteed dual-norm bound}\label{subsec:flux_bound}
We define
\begin{align}
\eta_{\mathrm{mis}}(\sigma_h)
&:= \|\sigma_h - A \nabla \tilde u_h \|_{A^{-1}},
\label{eq:eta_mis}\\
\eta_{\mathrm{osc}}(\sigma_h)
&:= \left(\sum_{T\in\mathbb T_h} \left(\frac{h_T}{\pi} \right)^2
 \|\mathcal R_T(\sigma_h)\|_{L^2(T)}^2 \right)^{\frac{1}{2}}.
\label{eq:eta_osc}
\end{align}

Here, ``mis'' stands for ``mismatch''. Thus, $\eta_{\mathrm{mis}}$ measures the mismatch between the reconstructed certificate flux $\sigma_h$ and the physical flux $A\nabla\tilde u_h$ associated with the conforming state.

In this section, we use the $L^2$-data setting specified in Section~\ref{sec:problem}. In particular, the elementwise residuals and their cellwise projections are interpreted as $L^2$-functions. Thus, the local quantities $\mathcal R_T(\sigma_h)$ appearing below are understood as $L^2(T)$-functions.

\begin{lemma}[Guaranteed residual bound from an equilibrated flux]\label{thm:res-majorant}
Let $\tilde u_h\in V_h$ be given and let $\sigma_h\in H(\mathrm{div};\Omega)$ satisfy the mean equilibration condition \eqref{eq:mean-equil}. Let $A(x)$ be the tensor used in  the energy norm $\|v\|_V = \left( \int_\Omega A \nabla v\cdot\nabla v dx \right)^{\frac{1}{2}}$, and assume the coercivity bound
\begin{align}
A(x)\xi\cdot\xi \geq \alpha_0|\xi|^2
\quad \text{for a.e. }x\in\Omega,\ \forall\,\xi\in\mathbb R^d, \label{eq:A0-coerc}
\end{align}
with some computable $\alpha_0>0$. Then, the dual norm of the residual satisfies
\begin{align}
\|\mathcal F(\tilde u_h)\|_{V^*}\;\leq\; \eta(\sigma_h)
:=  \eta_{\mathrm{mis}}(\sigma_h)+\alpha_0^{-1/2}\,\eta_{\mathrm{osc}}(\sigma_h). \label{eq:dual_res_bound}
\end{align}
Here, $\eta(\sigma_h)$ can be used as the certified residual bound $\mathfrak r$ required in {(C1)} of Section \ref{subsec:NK-assumptions}.
\end{lemma}

\begin{proof}
A proof can be found in Appendix	\ref{Appendix=A=3}.
\end{proof}

\begin{remark}[Explicit flux routes (Marini-type and CR--RT bridges)]
In low-order settings, Marini-type relations provide an inexpensive route to mixed quantities \cite{Mar85}. In three dimensions, explicit CR--RT correspondences on anisotropic meshes provide additional structure for constructing $H(\mathrm{div})$-conforming fluxes without relying on restrictive angle conditions \cite{IshKobTsu21b}.
\end{remark}

\section{Explicit equilibrated flux reconstruction (Marini-type)}\label{sec:fluxrecon}

This section explains how the certificate flux used in the residual bound is constructed in the low-order setting. The construction is based on an auxiliary CR problem and a Marini-type RT reconstruction. The aim is to obtain an explicit $H(\mathrm{div})$-conforming flux suitable for the residual estimate in Section~\ref{sec:flux}.

\subsection{Preliminaries}
Let $T \in \mathbb{T}_h$. For $k \in \mathbb{N}_0 := \mathbb{N} \cup \{ 0 \}$, $\mathbb{P}^k(T)$ is spanned by the restriction to $T$ of polynomials in $\mathbb{P}^k$, where  $\mathbb{P}^k$ denotes the space of polynomials with a maximum of $k$ degrees. We define the standard discontinuous finite-element space as
\begin{align*}
\displaystyle
V_{h}^{DC(0)} &:= \left\{ v_h \in {L^{\infty}(\Omega)}; \ v_h|_{T} \in \mathbb{P}^{0}({T}) \quad \forall T \in \mathbb{T}_h  \right\}. \label{dis=sp}
\end{align*}
The $L^2$-orthogonal projection onto $\mathbb{P}^0({T})$ is the linear operator ${\Pi}_{{T}}^{0}: L^1({T}) \to \mathbb{P}^{0}({T})$ defined as
\begin{align*}
\displaystyle
\int_{{T}} ({\Pi}_{{T}}^{0} {\varphi} - {\varphi})  dx = 0 \quad \forall {\varphi} \in L^1({T}). 
\end{align*}

We define a broken (piecewise) Hilbert space as
\begin{align*}
\displaystyle
H^1(\mathbb{T}_h) &:= \left\{ \varphi \in L^2(\Omega); \ \varphi|_{T} \in H^1(T) \ \forall T \in \mathbb{T}_h  \right\}.
\end{align*}
For any $v \in H^1(\mathbb{T}_h)$, we write $\nabla_h v$ for the broken gradient, defined by $(\nabla_h v)|_T := \nabla(v|_T)$ for each element $T\in\mathbb T_h$. If $v\in H^1(\Omega)$, then $\nabla_h v=\nabla v$ a.e.\ in $\Omega$.

Let $\mathcal{F}_h^i$ be the set of interior faces and let $\mathcal{F}_h^{\partial}$ denote the set of boundary faces. We set $\mathcal{F}_h := \mathcal{F}_h^{i}\cup \mathcal{F}_h^{\partial}$. If $F \in \mathcal{F}_h^i$ with $F = \partial T_{+} \cap \partial T_{-}$, $T_{+},T_{-} \in \mathbb{T}_h$, $+ > -$, let $n_F$ be the unit normal vector from $T_{+}$ to  $ T_{-}$. We once and for all fix a local ordering of the neighboring elements and use the symbol "$+ > -$" to indicate that $T_+$ is the first element and $T_-$ the second in this ordering. Suppose that $F \in \mathcal{F}_h^i$ with {$F = \partial T_{+} \cap \partial T_{-}$}, $T_{+},T_{-} \in \mathbb{T}_h$. For such a face, we set  $v_{+} := v{|_{T_{+}}}$ and $v_{-} := v{|_{T_{-}}}$ for any ${{v}} \in H^{1}(\mathbb{T}_h)^d$. The jump in the normal component is defined as
\begin{align*}
\displaystyle
&[\![{ {v}} \cdot { {n}}]\!] := [\![ { {v}} \cdot { {n}} ]\!]_F := { {v}_{+}} \cdot {{n}_F} - {{v}_{-}} \cdot {{n}_F}.
\end{align*}
For $T \in \mathbb{T}_h$, the local Raviart--Thomas (RT) polynomial space is defined as
\begin{align*}
\displaystyle
\mathbb{RT}^0(T) := \mathbb{P}^0(T)^d + {x} \mathbb{P}^0(T), \quad {x} \in \mathbb{R}^d.
\end{align*}
The RT finite-element space is defined as follows:
\begin{align*}
\displaystyle
V^{RT}_{h} &:= \{ {v_h} \in L^1(\Omega)^d: \  {v_h}|_T \in \mathbb{RT}^0(T), \ \forall T \in \mathbb{T}_h, \  [\![ {v_h} \cdot {n} ]\!]_F = 0, \ \forall F \in \mathcal{F}_h^i \}.
\end{align*}
Let $V_h^{CR}$ denote the lowest-order CR space, and let $V_{h,0}^{CR}\subset V_h^{CR}$ be its homogeneous Dirichlet subspace defined as
\begin{align*}
\displaystyle
V_{h,0}^{CR}
:= \left \{ v_h \in V_h^{CR}: \  \int_F v_h ds=0 \ \forall F \in \mathcal F_h^{\partial} \right \}.
\end{align*}

The guaranteed residual bound in Lemma~\ref{thm:res-majorant} requires a flux $\sigma_h\in H(\mathrm{div};\Omega)$ and the elementwise mean equilibration \eqref{eq:mean-equil}. To this end, we write, for each $T\in\mathbb T_h$,
\begin{equation}
\mathcal R_T(\sigma_h) := s_h - \nabla\!\cdot\sigma_h \quad \text{in }T,
\label{eq:RT-unified}
\end{equation}
where the (known) source term $s_h$ is defined from the computed approximation $\tilde u_h$ as
\begin{align}
s_h := c(\cdot,\tilde u_h) - f \quad  \text{(semilinear diffusion--reaction)}. \label{eq:gh-def}
\end{align}

We define $r_h\in V_h^{DC(0)}$ as
\begin{equation}
r_h|_T := \Pi_T^0(s_h)\quad \forall\,T\in\mathbb T_h.
\label{eq:rh-def}
\end{equation}

\begin{lemma}[Divergence target implies mean equilibration]\label{lem:div-implies-mean}
Let $\sigma_h\in H(\mathrm{div};\Omega)$ satisfy
\begin{equation}
\nabla\!\cdot\sigma_h = r_h \quad \text{in each }T\in\mathbb T_h.
\label{eq:div-target}
\end{equation}
Then, \eqref{eq:mean-equil} holds.
\end{lemma}

\begin{proof}
By \eqref{eq:RT-unified} and \eqref{eq:div-target},
\begin{align*}
\displaystyle
\int_T \mathcal R_T(\sigma_h) dx
= \int_T s_h dx - \int_T r_h dx.
\end{align*}
Because $r_h|_T=\Pi_T^0(s_h)$ preserves the element mean, the right-hand side vanishes.
\end{proof}

\subsection{Marini-type explicit construction} \label{subsec:flux:marini}

We describe an explicit low-order route to a certificate flux, based on a Marini-type correspondence between nonconforming and mixed formulations. The key point is that the construction produces $\sigma_h\in H(\operatorname{div};\Omega)$ and enforces \eqref{eq:div-target} without solving patchwise mixed saddle-point reconstruction problems; the price is one auxiliary low-order CR problem.

This structure is useful in the anisotropic low-order setting considered here: the flux certificate is obtained from a standard scalar nonconforming problem and an explicit cellwise $\mathbb{RT}^0$ formula. The anisotropic behavior is subsequently monitored via the flux-mismatch term within the residual bound. Section~\ref{sec5=3} provides a consistency statement for this term across anisotropic mesh families.

For higher-order conforming states, the lowest-order CR--RT reconstruction still gives a valid residual bound under the hypotheses of Lemma~\ref{thm:res-majorant}, but may not retain the accuracy or convergence order of the higher-order approximation. Higher-order RT or Brezzi--Douglas--Marini reconstructions may be needed for sharper bounds, together with a corresponding anisotropic analysis. We leave this extension to future work.

In cases of genuinely variable diffusion tensors, substituting $A$ with an elementwise constant tensor $A_{0,T}$ can result in an additional coefficient-approximation contribution to the residual bound. Consequently, the current construction is not considered a general higher-order variable-coefficient equilibrated reconstruction.

\begin{theorem}[Marini-type certificate flux for piecewise-constant tensor $A$]\label{lem:marini-unified}
Let $A$ be symmetric, positive definite. We set $A_{0,T} := \Pi_T^0 A$, which is constant on each $T\in\mathbb T_h$. Let $\tilde u_h^{CR}\in V_{h,0}^{CR}$ be the solution of the problem
\begin{align}
\displaystyle
\sum_{T \in \mathbb{T}_h} \int_T A_{0,T} \nabla \tilde u_h^{CR}\cdot \nabla v_h dx
= \int_\Omega \ell_h\, v_h dx
\quad \forall v_h \in V_{h,0}^{CR}, \label{eq:CR-ref}
\end{align}
where
\begin{align}
\displaystyle
\ell_h|_T := -\,r_h|_T =  - \Pi_T^0 (s_h). \label{eq:ellh-def}
\end{align}
For each element $T$ with barycenter $x_T$, we define
\begin{align}
\displaystyle
\sigma_T^*(x)
:= A_{0,T}\,\nabla_h \tilde u_h^{CR}|_T + \frac{r_h|_T}{d}\,(x-x_T),
\quad x\in T. \label{eq:sigma-star-lem}
\end{align}
We define a piecewise flux $\sigma_h \in L^1(\Omega)^d$ by prescribing, for every face  $F\subset\partial T$, the $\mathbb{RT}^0(T)$ face flux degrees of freedom
\begin{align}
\displaystyle
\int_F (\sigma_h|_T)\cdot n_{F} ds
:= \int_F \sigma_T^*\cdot n_{F} ds,
\quad \forall\,F\subset\partial T, \label{eq:rt0-dof-local}
\end{align}
and taking $\sigma_h|_T\in \mathbb{RT}^0(T)$ as the unique field with these degrees of freedom. Since $\sigma_T^*\in\mathbb{RT}^0(T)$ and has exactly the same face-flux degrees of freedom as $\sigma_h|_T$, unisolvence gives $\sigma_h|_T=\sigma_T^*$ on every element $T$. Then, the piecewise flux $\sigma_h$ satisfies:
\begin{enumerate}
 \item For an interior face $F \in \mathcal{F}_h^i$ with $F = \partial T_{+} \cap \partial T_{-}$, $T_{+},T_{-} \in \mathbb{T}_h$,
\begin{align}
\displaystyle
\int_F [\![ \sigma_h \cdot n ]\!] ds  = \int_F (\sigma_h|_{T_+}) \cdot n_{F} ds
- \int_F (\sigma_h|_{T_-})\cdot n_{F} ds = 0. \label{sec5=7}
\end{align}
Consequently, $\sigma_h \in V_h^{RT}$, because $\sigma_h \cdot n_F \in \mathbb{P}^0(F)$ implies $ [\![ \sigma_h \cdot n ]\!] = 0$ a.e. on $F$.
 \item For each element $T \in \mathbb T_h$,
\begin{align}
\displaystyle
\nabla \cdot (\sigma_h|_T) = r_h|_T.  \label{sec5=8}
\end{align}
Consequently, $\int_T \mathcal R_T(\sigma_h) dx=0$ for all $T\in\mathbb T_h$.
\end{enumerate}

\end{theorem}

\begin{proof}
For an interior face $F \in \mathcal{F}_h^i$ with $F = \partial T_{+} \cap \partial T_{-}$, $T_{+},T_{-} \in \mathbb{T}_h$, let $\varphi_F\in V_{h,0}^{CR}$ be the CR  basis function associated with $F$, characterized by
\begin{itemize}
 \item  $\varphi_F$ is supported on the patch $\omega_F := T_+ \cup T_-$,
 \item $\varphi_F$ is affine on each of $T_\pm$,
 \item the face  means satisfy
\begin{align*}
\displaystyle
\frac{1}{|F|_{d-1}}\int_F \varphi_F ds = 1, \quad \frac{1}{|F^{\prime}|_{d-1}}\int_{F^{\prime}} \varphi_F ds = 0 \quad \text{for any other face $F^{\prime} \neq F$ of $T_\pm$},
\end{align*}
and $\int_F \varphi_F ds = 0$ on $F \in \mathcal{F}_h^{\partial}$. 
\end{itemize}
Because $A_{0,T}$ is constant on each element $T$ and $\tilde u_h^{CR}$ is affine on each element, the vector field $A_{0,T} \nabla_h \tilde u_h^{CR}|_T$ is constant on $T$. Thus,
\begin{align*}
\displaystyle
\nabla  \cdot \left (A_{0,T}\nabla_h \tilde u_h^{CR}|_T \right) =0
\quad \text{in each $T \in \mathbb T_h$}.
\end{align*}
Furthermore, on any face $F \subset \partial T$, the scalar $(A_{0,T}\nabla_h \tilde u_h^{CR}|_T) \cdot n_{F}$ is constant on $F$.

Substituting $\varphi_F$ for $v_h$ in \eqref{eq:CR-ref},
\begin{align}
\displaystyle
\int_{T_+} A_{0,T_+} \nabla \tilde u_h^{CR}|_{T_+} \cdot \nabla_h \varphi_F dx + \int_{T_-} A_{0,T_-} \nabla \tilde u_h^{CR}|_{T_-} \cdot \nabla_h \varphi_F dx
= - \int_{T_+ \cup T_-} r_h  \varphi_F dx. \label{eq:test-CR}
\end{align}
On each of $T_\pm$,  by integrating by parts and $\nabla\!\cdot(A_{0,T}\nabla_h \tilde u_h^{CR}|_T)=0$,
\begin{align}
\displaystyle
\left[ \left(A_{0,T_+}\nabla_h \tilde u_h^{CR}|_{T_+}\cdot n_{F}\right)
-\left(A_{0,T_-}\nabla_h \tilde u_h^{CR}|_{T_-}\cdot n_{F}\right) \right] \int_F \varphi_F\,ds
=
-\int_{T_+ \cup T_-} r_h \varphi_F dx. \label{eq:flux-balance-A0}
\end{align}
Let $T\in\{T_+,T_-\}$. Because $r_h|_T$ is constant and $\varphi_F|_T$ is affine, it holds that
\begin{align}
\displaystyle
\int_T r_h \varphi_F dx
=
\frac{r_h|_T}{d} \int_{\partial T} \varphi_F (x-x_T)\cdot n_T ds, \label{eq:key-identity}
\end{align}
where $n_T$ denotes the unit outward normal on $\partial T$. Indeed, using $\nabla\!\cdot(x-x_T)=d$ and the product rule: $\nabla\!\cdot(\varphi_F(x-x_T))=(x-x_T)\cdot\nabla\varphi_F + d\,\varphi_F$,
\begin{align*}
\displaystyle
\int_T r_h \varphi_F dx
&=\frac{r_h|_T}{d}\int_T \varphi_F \nabla \cdot(x-x_T) dx \\
&=\frac{r_h|_T}{d}\int_T \nabla \cdot\bigl(\varphi_F(x-x_T)\bigr) dx
-\frac{r_h|_T}{d}\int_T (x-x_T)\cdot \nabla \varphi_F dx.
\end{align*}
The last term vanishes because $\nabla\varphi_F$ is constant on $T$ and $\int_T (x-x_T)\,dx=0$ by the definition of the barycenter $x_T$. Therefore, \eqref{eq:key-identity} follows by the divergence theorem.

It holds $\int_{F^{\prime}} \varphi_F ds = 0$ for any other face $F^{\prime} \neq F$ of $T$, and $(x-x_T) \cdot n_{T,F^*}$ is constant on each face $F^* \subset \partial T$, where the normal is constant, and we write it as $n_{T,F}$. Therefore,
\begin{align*}
\displaystyle
\int_{\partial T} \varphi_F (x-x_T)\cdot n_T ds
=
\int_{F} \varphi_F (x-x_T)\cdot n_{T,F}ds
=
\bigl((x_F-x_T)\cdot n_{T,F}\bigr) \int_F \varphi_F ds,
\end{align*}
where $x_F$ denotes the barycenter of $F$. Thus,
\begin{align}
\displaystyle
\int_T r_h \varphi_F dx
=
\frac{r_h|_T}{d} \left((x_F-x_T)\cdot n_{T,F}\right) \int_F \varphi_F ds. \label{eq:vol-to-face}
\end{align}
Using \eqref{eq:vol-to-face} for $T=T_+$ and $T=T_-$, we have
\begin{align}
\displaystyle
\int_{T_+ \cup T_-} r_h \varphi_F dx
&= \left[ \frac{r_h|_{T_+}}{d} \left((x_F-x_{T_+})\cdot n_{T_{+},F}\right) + \frac{r_h|_{T_-}}{d} \left((x_F-x_{T_-})\cdot n_{T_-,F}\right)  \right] \int_F \varphi_F ds. \label{sec5=13}
\end{align}
Recalling the definition \eqref{eq:sigma-star-lem}:
\begin{align*}
\displaystyle
\sigma_T^*(x)=A_{0,T}\nabla_h \tilde u_h^{CR}|_T+\frac{r_h|_T}{d}(x-x_T),
\end{align*}
and using \eqref{eq:rt0-dof-local}, \eqref{eq:flux-balance-A0}, \eqref{sec5=13} and $\int_F \varphi_F ds = |F|_{d-1}$, we have
\begin{align*}
\displaystyle
\int_F [\![ \sigma_h \cdot n ]\!] ds  
&= \int_F (\sigma_h|_{T_+}) \cdot n_{T_+,F} ds
+ \int_F (\sigma_h|_{T_-})\cdot n_{T_-,F} ds \\
&= \int_F \sigma_{T_+}^* \cdot n_{T_+,F} ds
+ \int_F \sigma_{T_-}^*\cdot n_{T_-,F} ds \\
&= \left[ \left(A_{0,T_+}\nabla_h \tilde u_h^{CR}|_{T_+}\cdot n_{F}\right)
-\left(A_{0,T_-}\nabla_h \tilde u_h^{CR}|_{T_-}\cdot n_{F}\right) \right] \int_F \varphi_F ds \\
&\quad + \int_F \left[ \frac{r_h|_{T_+}}{d} \left((x-x_{T_+})\cdot n_{T_{+},F}\right) + \frac{r_h|_{T_-}}{d} \left((x-x_{T_-})\cdot n_{T_-,F}\right)  \right]  ds \\
&= \left[ \left(A_{0,T_+}\nabla_h \tilde u_h^{CR}|_{T_+}\cdot n_{F}\right)
-\left(A_{0,T_-}\nabla_h \tilde u_h^{CR}|_{T_-}\cdot n_{F}\right) \right] \int_F \varphi_F ds \\
&\quad + \int_{T_+ \cup T_-} r_h \varphi_F dx \\
&= 0.
\end{align*}
In particular, the normal component is single-valued across interior faces and therefore $\sigma_h\in H(\mathrm{div};\Omega)$, because $\sigma_h|_T\in \mathbb{RT}^0(T)$ by construction, we conclude $\sigma_h \in V_h^{RT}$.

Let $T \in \mathbb T_h$. By construction, because $\sigma_h|_T\in \mathbb{RT}^0(T)$, $\nabla\!\cdot(\sigma_h|_T)$ is constant on $T$. Using the divergence theorem and the definition of the face flux degrees of freedom,
\begin{align*}
\displaystyle
\int_T \nabla \cdot(\sigma_h|_T) dx
=\int_{\partial T} (\sigma_h|_T)\cdot n_T ds
=\int_{\partial T} \sigma_T^*\cdot n_T ds
=\int_T \nabla \cdot \sigma_T^* dx.
\end{align*}
Because $A_{0,T}\nabla_h \tilde u_h^{CR}|_T$ is constant, its divergence vanishes, and $\nabla \cdot(x-x_T)=d$, we have
\begin{align*}
\displaystyle
\nabla\!\cdot\sigma_T^*
=\nabla \cdot \left(A_{0,T}\nabla_h \tilde u_h^{CR}|_T \right)
+\nabla \cdot \left(\frac{r_h|_T}{d}(x-x_T) \right)
= r_h|_T \quad \text{in $T$},
\end{align*}
which leads to
\begin{align*}
\displaystyle
\int_T \nabla \cdot(\sigma_h|_T) dx = r_h|_T |T|_d.
\end{align*}
Because $\nabla\!\cdot(\sigma_h|_T)$ is constant on $T$, this implies $\nabla \cdot(\sigma_h|_T)=r_h|_T$ pointwise on $T$.	
\end{proof}

\begin{remark}
The additional solve \eqref{eq:CR-ref} is low-order, linear, and introduced solely to produce an explicit $H(\mathrm{div};\Omega)$ certificate flux for the tensor $A$. We emphasize that our subsequent certification arguments require only the two outcomes $\sigma_h \in H(\mathrm{div};\Omega)$ and $\nabla \cdot \sigma_h=r_h$ (elementwise), and do not rely on a full equivalence between nonconforming and mixed formulations.
\end{remark}

\subsection{A consistency statement on anisotropic meshes} \label{sec5=3}
The preceding construction gives the certificate flux used in the dual residual bound. For $A=I$, the first term in this bound is the flux-mismatch term, defined by
\begin{align*}
  \eta_{\rm mis}(\sigma_h)
  :=
  \|\sigma_h-\nabla\tilde u_h\|_{L^2(\Omega)^d}.
\end{align*}
We record when this term is consistent along anisotropic mesh families. The following result is restricted to the low-order Poisson setting. It is not an anisotropic efficiency theorem for general equilibrated estimators; instead, it separates the algebraic CR--RT reconstruction from the finite element energy-error estimate for the underlying approximation.

\begin{proposition}[Consistency of the flux-mismatch term on anisotropic meshes]
\label{prop:anisotropic-mismatch-consistency}
Assume that $A=I$ and consider the Poisson problem
\begin{align*}
\displaystyle
  -\varDelta u=f\quad\text{in }\Omega,\quad u=0\quad\text{on }\partial\Omega.
\end{align*}
We write $\Pi_h^0$ for the elementwise $L^2$-projection, $(\Pi_h^0 v)|_T:=\Pi_T^0(v|_T)$. Let $\tilde u_h\in V_h$ be a conforming finite element approximation of $u$. Let $\tilde u_h^{CR}\in V_{h,0}^{CR}$ be the lowest-order CR approximation with the elementwise constant right-hand side $f_h:=\Pi_h^0 f$, and let $\sigma_h \in V_h^{RT}$ be the Marini-type reconstructed flux
\begin{align*}
\displaystyle
\sigma_h|_T
=
\nabla_h \tilde u_h^{CR}|_T
-
\frac{f_h|_T}{d}(x-x_T),
\quad T \in \mathbb T_h,
\end{align*}
where $x_T$ denotes the barycenter of $T$. Then,
\begin{align*}
\displaystyle
\|\sigma_h-\nabla \tilde u_h\|_{L^2(\Omega)^d}
\leq
\|\nabla_h \tilde u_h^{CR}-\nabla u\|_{L^2(\Omega)^d}
+
\|\nabla(u-\tilde u_h)\|_{L^2(\Omega)^d}
+
h\|f_h\|_{L^2(\Omega)}.
\end{align*}
Consequently, suppose that $h\to0$ along the anisotropic mesh family and that
\begin{align*}
\displaystyle
\|\nabla_h \tilde u_h^{CR}-\nabla u\|_{L^2(\Omega)^d} \to 0, \quad \|\nabla(u-\tilde u_h)\|_{L^2(\Omega)^d} \to 0,
\end{align*}
Then,
\begin{align*}
\displaystyle
\|\sigma_h-\nabla \tilde u_h\|_{L^2(\Omega)^d}\to0.
\end{align*}
In particular, the low-order CR--RT reconstruction is consistent with the flux-mismatch term whenever the underlying CR and conforming approximations are energy-consistent on the anisotropic mesh family under consideration.
\end{proposition}

\begin{proof}
From the Marini-type formula and the triangle inequality,
\begin{align*}
\displaystyle
\|\sigma_h-\nabla \tilde u_h\|_{L^2(\Omega)^d}
&\leq
\|\nabla_h \tilde u_h^{CR}-\nabla u\|_{L^2(\Omega)^d}
+
\|\nabla(u-\tilde u_h)\|_{L^2(\Omega)^d}\\
&\quad+
\left(
\sum_{T\in\mathbb T_h}
\left\|
\frac{f_h|_T}{d}(x-x_T)
\right\|_{L^2(T)^d}^2
\right)^{1/2}.
\end{align*}
Because $|x-x_T| \leq h_T \leq h$ on each element $T$, we have
\begin{align*}
\displaystyle
\left\|
\frac{f_h|_T}{d}(x-x_T)
\right\|_{L^2(T)^d}
\leq
\frac{h_T}{d}\|f_h\|_{L^2(T)}
\leq
h \|f_h\|_{L^2(T)}.
\end{align*}
Summing over all elements gives
\begin{align*}
\displaystyle
\left(
\sum_{T\in\mathbb T_h}
\left\|
\frac{f_h|_T}{d}(x-x_T)
\right\|_{L^2(T)^d}^2
\right)^{1/2}
\leq h \|f_h\|_{L^2(\Omega)}.
\end{align*}
Furthermore, because $f_h = \Pi_h^0 f$ and $\Pi_h^0$ is $L^2$-stable,
\begin{align*}
\displaystyle
h\|f_h\|_{L^2(\Omega)}
\leq h \|f \|_{L^2(\Omega)}
\to 0
\quad \text{as $h\to0$}
\end{align*}
whenever $f\in L^2(\Omega)$. Therefore, under the assumed energy consistency of the CR approximation and the conforming state, the flux-mismatch term satisfies
\begin{align*}
\displaystyle
\|\sigma_h-\nabla \tilde u_h \|_{L^2(\Omega)^d}\to 0.
\end{align*}
\end{proof}

\begin{remark} \label{rem:conditional-anisotropic-consistency}
Proposition~\ref{prop:anisotropic-mismatch-consistency} is conditional and does not independently establish an anisotropic energy-error estimate for either the CR approximation or the conforming approximation. Instead, it demonstrates that when these two approximations achieve energy consistency on a specified anisotropic mesh family, the CR--RT reconstruction does not introduce additional complications in the flux-mismatch term. The assumptions outlined in Proposition~\ref{prop:anisotropic-mismatch-consistency} are substantive. In the context of the Poisson problem, the energy convergence of the CR approximation can be demonstrated on anisotropic mesh families that meet certain semi-regular geometric conditions. Similarly, for conforming Lagrange approximations, energy convergence is achieved under suitable anisotropic interpolation estimates and compatibility conditions concerning the mesh family and the exact solution. These considerations are part of anisotropic finite element approximation theory; the algebraic CR--RT reconstruction itself is kept separate; refer to \cite{IshizakaKobayashiTsuchiya2023,IshizakaPhD2022} for further details.

\end{remark}

\section{Certifying (C2): a computable stability constant for the linearization} \label{sec:stability}

This section treats the stability estimate required for the second verification condition (C2). In the coercive setting considered in this paper, the linearized operator is bounded from below in the energy norm, which gives a computable lower bound for the stability constant entering the NK criterion.

\subsection{Stability constant required by NK}\label{subsec:C2-stability}
Let $\tilde u_h\in V_h$ be the computed approximation. Recall that
\begin{align*}
\displaystyle
\mathcal L_{\tilde u_h} = D\mathcal F(\tilde u_h)\in \mathcal L(V,V^*).
\end{align*}
We denote by
\begin{align}
\displaystyle
B_{\tilde u_h}(\delta u,v) := \langle \mathcal L_{\tilde u_h} \delta u,  v\rangle \quad \forall \delta u,v\in V \label{eq:lin-bilinear}
\end{align}
the bilinear form induced by the linearization. In the diffusion-reaction context addressed herein, $B_{\tilde u_h}$ is typically {symmetric}. Consequently, we focus on the symmetric elliptic case and establish (C2) through a coercivity constant. Although more general (nonsymmetric) linearizations can be addressed using an inf-sup theory, this extension is not pursued in the current paper.

\noindent\textbf{Coercivity constant.}
Let $\alpha>0$ be a computable number such that
\begin{align}
\displaystyle
B_{\tilde u_h}(v,v) \geq \alpha \|v\|_V^2 \quad \forall v\in V. \label{eq:C2-coerc}
\end{align}
We write the optimal coercivity constant as
\begingroup
\begin{align}
\displaystyle
\alpha_* := \inf_{v\in V\setminus \{0\}}\frac{B_{\tilde u_h}(v,v)}{\|v\|_V^2}>0,
\quad 0<\alpha\leq\alpha_* . \label{eq:alpha-def}
\end{align}
\endgroup
The verification only requires a computable lower bound $\alpha$, not the exact value of $\alpha_*$.
From the Lax--Milgram lemma, \eqref{eq:C2-coerc} implies that $\mathcal L_{\tilde u_h} :V\to V^*$ is bijective and
\begin{align}
\displaystyle
\|\mathcal L_{\tilde u_h}^{-1}\|_{\mathcal L(V^*,V)} \leq \alpha^{-1}, \label{eq:C2-invbound}
\end{align}
which is exactly the stability ingredient required in condition (C2).

\subsection{A monotone (coercive) bound}\label{subsec:monotone-bound}
We consider inexpensive coercivity bounds that apply in monotone settings.

For the operator $\mathcal F$ defined in Section \ref{subsec:operator-linearization}, one has
\begin{align}
\displaystyle
B_{\tilde u_h}(v,v)
=  \int_\Omega A\nabla v\cdot\nabla v dx +\int_\Omega (\partial_s c) (x,\tilde u_h) v^2 dx,\label{eq:semilinear-Bvv}
\end{align}
where $\partial_s c(x,w) = b(x,w) + w \partial_s b(x,w)$. Here, $\partial_s$ denotes the derivative with respect to the scalar argument $s$.

If the monotonicity condition $(\partial_s c) (x,\tilde u_h) \geq  0$ a.e. in $\Omega$, then
\begin{align*}
\displaystyle
B_{\tilde u_h}(v,v) \geq \int_\Omega A\nabla v\cdot\nabla v dx .
\end{align*}
In particular, if the energy norm is taken as $\displaystyle \|v\|_V = \left( \int_\Omega A \nabla v\cdot\nabla v dx \right)^{\frac{1}{2}}$, then \eqref{eq:C2-coerc} holds with
\begin{align}
\displaystyle
\alpha = 1. \label{eq:alpha-ge-1}
\end{align}

\begin{remark}
When $(\partial_s c)(x,\tilde u_h)$ changes sign, one may still attempt a purely analytic lower bound by estimating the negative part of $\int_\Omega (\partial_s c)v^2$ in terms of $\|\nabla v\|_{L^2(\Omega)}^2$ via a Poincar\'e inequality. Such bounds are domain-dependent and can be very conservative. Furthermore, they do not cover genuinely non-coercive or nearly singular linearizations. These cases require a separate invertibility verification and are outside the scope of the present paper.

\end{remark}

\section{Certifying (C3): Lipschitz bound for the linearization}\label{sec:C3}

This section derives the Lipschitz bound required for the nonlinear remainder in the NK argument. The estimates are stated in terms of computable scalar quantities so that the final verification conditions reduce to the evaluation of the residual size, the stability constant, and the nonlinear Lipschitz bound.

\subsection{What (C3) requires}\label{subsec:C3-require}
Section~\ref{sec:NK} reduces the verification step to a small number of scalar inequalities. Besides the residual bound and a certified stability bound for the linearization, we still need a computable Lipschitz bound for the derivative $D \mathcal F$ on the verification ball (condition (C3)), and bounds controlling the variation of a quantity of interest $\mathcal J(u)$ on that ball.  

Recall the closed ball $B_{\rho} \subset\mathcal U$ from
Sections \ref{subsec:U-and-balls} and \ref{subsec:NK-assumptions}, and
\begin{align*}
\displaystyle
\mathcal L_w = D\mathcal F(w)\in \mathcal L(V,V^*) \quad \forall w\in\mathcal U.
\end{align*}
Condition {(C3)} in Theorem~\ref{thm:NK} requires a computable function
$L:(0,\bar\rho]\to(0,\infty)$ such that, for any $0<\rho\leq\bar\rho$,
\begin{align} \label{eq:C3-Lip}
\|\mathcal L_w-\mathcal L_z\|_{\mathcal L(V,V^*)}
\le L(\rho)\,\|w-z\|_{V}
\quad \forall\,w,z\in B_{\rho}.
\end{align}
This section provides inexpensive, fully computable bounds for $L(\rho)$. We present a semilinear route that is sufficient for the main body. 

\subsection{Embedding constants} \label{subsec:embed}

Assume that there exists a positive constant $C_p$ satisfying the norm bound
\begin{align}
\displaystyle
\|v\|_{L^p(\Omega)} &\le C_p \|v\|_V \quad \forall v \in V. \label{eq:CP}
\end{align}

Concrete choices of such constants used in the numerical experiments are specified in Section~\ref{sec:numerics}.

In applications, we treat $C_p$ as a computable input. If there exists a computable $\alpha_0 > 0$ such that
\begin{align*}
\displaystyle
A (x)\xi\cdot\xi\ge \alpha_0|\xi|^2 \quad \text{for a.e.\ }x\in\Omega,\ \forall\,\xi\in\mathbb R^d,
\end{align*}
then
\begin{align*}
\displaystyle
\| v \|_V^2 = \int_\Omega A \nabla v\cdot\nabla v dx \geq \alpha_0 \| \nabla v \|^2_{L^2(\Omega)},
\end{align*}
which leads to
\begin{align*}
\displaystyle
 \| \nabla v \|_{L^2(\Omega)} \leq \alpha_0^{- \frac{1}{2}} \| v \|_V.
\end{align*}
From the Sobolev's embedding theorem $H_0^1(\Omega) \hookrightarrow L^p(\Omega)$, there exists a positive constant $C_{s,p}$ such that 
\begin{align*}
\displaystyle
\| v \|_{L^p(\Omega)} \leq C_{s,p} \| \nabla v \|_{L^2(\Omega)} \quad \forall v \in H_0^1(\Omega),
\end{align*}
with
\begin{align*}
\displaystyle
p \in [2,\infty) \quad \text{if $d =2$}, \quad p \in [2,6] \quad \text{if $d =3$}.
\end{align*}
Therefore, 
\begin{align*}
\displaystyle
\| v \|_{L^p(\Omega)} \leq C_{s,p} \alpha_0^{- \frac{1}{2}} \| v \|_V.
\end{align*}

\begin{remark}
Note that the scalar checks require numerical values of the embedding constants.  Here, the phrase ``scalar checks'' refers to the finite numerical inequalities in the NK criterion, such as $p(\rho) \leq 0$ and $q(\rho) < 1$, after the residual bound, stability lower bound, and Lipschitz bound have been computed.

These checks require numerical values of the embedding constants (e.g., $C_6$). We do not need sharp constants: any guaranteed upper bound is sufficient, and these constants are treated as domain-dependent inputs that can be precomputed once for a given $\Omega$.
\end{remark}

\begin{remark}[Computable Poincar\'e and Sobolev embedding constants] \label{rem=embed}
For $d=3$, extending v by zero to $\mathbb{R}^3$ yields the critical Sobolev inequality with the sharp Aubin--Talenti constant \cite{Aub76,Tal76}. For $d=2$ and fixed $p<\infty$, verified enclosures of the optimal constants on bounded convex domains are available; see \cite{TanSekMizOis17} and references therein.
\end{remark}

\subsection{Semilinear models: a general bound via coefficient derivatives}\label{subsec:C3-semi-general}
We consider the semilinear diffusion--reaction setting
\begin{align*}
\displaystyle
\langle \mathcal F(u),v\rangle
= \int_\Omega A\nabla u\cdot\nabla v dx +\int_\Omega c(x,u) v dx - \int_\Omega f v dx
\quad \forall v\in V,
\end{align*}
where $c(\cdot,\cdot)$ is the reaction term. Then, for any $w\in\mathcal U$ and any $\delta u,v\in V$,
\begin{align}\label{eq:C3-semi-lin}
\langle \mathcal L_w \delta u, v\rangle
= \int_\Omega A\nabla \delta u\cdot\nabla v dx
+ \int_\Omega \partial_s c(x,w)\,\delta u\,v dx.
\end{align}

\begin{lemma}[Semilinear (C3) bound]\label{lem:C3-semi}
Let $p\in[2,\infty)$ be such that the embedding $V\hookrightarrow L^p(\Omega)$ holds with a
known constant $C_p>0$ in \eqref{eq:CP}. Let $w,z\in\mathcal U$. Then,
\begin{align}\label{eq:C3-semi-op}
\|\mathcal L_w-\mathcal L_z\|_{\mathcal L(V,V^*)}
\leq C_p^2 \|\partial_s c(\cdot,w)-\partial_s c(\cdot,z)\|_{L^q(\Omega)},
\end{align}
where $q:=\frac{p}{p-2}\in[1,\infty]$ (with the convention $q=\infty$ when $p=2$). In particular, if there exists a computable constant $\Gamma(\rho)\ge0$ such that
\begin{align}\label{eq:C3-semi-gamma}
\|\partial_s c(\cdot,w)-\partial_s c(\cdot,z)\|_{L^q(\Omega)}
\leq \Gamma(\rho) \|w-z\|_V
\quad \forall\,w,z\in B_{\rho},
\end{align}
then \eqref{eq:C3} holds with $L(\rho)=C_p^2\,\Gamma(\rho)$.
\end{lemma}

\begin{proof}
The diffusion part cancels in $\mathcal L_w-\mathcal L_z$. Therefore, for any $\delta u,v\in V$,
\begin{align*}
\displaystyle
\langle (\mathcal L_w-\mathcal L_z)\delta u,v\rangle
= \int_\Omega (\partial_s c(\cdot,w)-\partial_s c(\cdot,z))\,\delta u\,v dx.
\end{align*}
The H\"older inequality with exponents $(q,p,p)$ and \eqref{eq:CP} gives
\begin{align*}
\displaystyle
|\langle (\mathcal L_w-\mathcal L_z)\delta u,v\rangle|
&\leq \|\partial_s c(\cdot,w)-\partial_s c(\cdot,z)\|_{L^q(\Omega)} \|\delta u\|_{L^p(\Omega)} \|v\|_{L^p(\Omega)} \\
&\leq C_p^2 \|\partial_s c(\cdot,w)-\partial_s c(\cdot,z)\|_{L^q(\Omega)} \|\delta u\|_V \|v\|_V.
\end{align*}
Taking the supremum over $\|\delta u\|_V=\|v\|_V=1$ yields \eqref{eq:C3-semi-op}. The last statement follows immediately from \eqref{eq:C3-semi-gamma}.
\end{proof}

\begin{example}[Allen--Cahn type reaction] \label{ex:AC-reaction}
We consider
\begin{align}\label{eq:AC-reaction}
c(x,u)=\kappa(x)\,(u^3-u), \quad \kappa \in L^\infty(\Omega),\ \kappa\geq 0 \ \text{for a.e. $x \in \Omega$}.
\end{align}
Then, $\partial_s c(x,u)=\kappa(x)\,(3u^2-1)$. Choose $p=4$ (hence $q=2$) and assume \eqref{eq:CP} holds with $C_4$.

For any $w,z\in B_{\rho}$,
\begin{align*}
\displaystyle
\partial_s c(x,w)-\partial_s c(x,z)=3\kappa(x) (w+z) (w-z),
\end{align*}
the H\"older inequality and \eqref{eq:CP} yield
\begin{align*}
\displaystyle
\|\partial_s c(\cdot,w)-\partial_s c(\cdot,z)\|_{L^2(\Omega)}
&\leq 3\|\kappa\|_{L^\infty(\Omega)} \|w+z\|_{L^4(\Omega)} \|w-z\|_{L^4(\Omega)} \\
&\leq 3 C_4^2 \|\kappa\|_{L^\infty(\Omega)} \|w+z\|_{V} \|w-z\|_{V} \\
&\leq 6 C_4^2 \|\kappa\|_{L^\infty(\Omega)} (\|\tilde u_h\|_V+\rho) \|w-z\|_{V}.
\end{align*}
where we used the fact that
\begin{align*}
\displaystyle
\|w\|_V \leq \|\tilde u_h\|_V + \rho,\quad \|z\|_V \leq \|\tilde u_h\|_V+\rho.
\end{align*}
Therefore, \eqref{eq:C3-semi-gamma} holds with
\begin{align*}
\displaystyle
\Gamma(\rho)=6 C_4^2 \|\kappa\|_{L^\infty(\Omega)}(\|\tilde u_h\|_V+\rho),
\end{align*}
which leads to
\begin{align*}
\displaystyle
\|\mathcal L_w-\mathcal L_z\|_{\mathcal L(V,V^*)}
\leq L_{\mathrm{ac}} \|w-z\|_V, \quad L_{\mathrm{ac}} := C_p^2 \Gamma(\rho).
\end{align*}
Therefore, the Lipschitz condition \eqref{eq:C3-Lip} holds with the computable choice $L(\rho)=L_{\mathrm{ac}}$.
\end{example}

\section{Verified outputs: guaranteed enclosures for quantities of interest}\label{sec:outputs}

Assume that the NK verification in Theorem \ref{thm:NK} succeeds for some radius $\rho>0$. Therefore, there exists a (locally unique) exact solution $u\in B_\rho\subset\mathcal U$. Let $\mathcal J:\mathcal U\to\mathbb R$ be a volume-type quantity of interest whose derivative admits an $L^2(\Omega)$-density; see Section~\ref{subsec:qoi}. This section provides certified enclosures for $\mathcal J(u)$.

\subsection{A baseline enclosure from a uniform derivative bound}\label{subsec:qoi:baseline}

\begin{assumption}[Uniform bound for the QoI derivative on $B_{\rho}$] \label{ass:qoi:DJ-unif}
The functional $\mathcal J$ is Fr\'echet differentiable on $B_{\rho}$, and there exists a computable constant $L_{\mathcal J}(\rho)>0$ such that
\begin{align}\label{eq:qoi:DJ-unif}
\|D \mathcal J(w)\|_{\mathcal L(V,\mathbb R)}\leq L_{\mathcal J}(\rho) \quad \forall w \in B_{\rho}.
\end{align}
\end{assumption}

\begin{proposition}[Certified baseline enclosure]\label{prop:qoi:baseline}
Under Assumption~\ref{ass:qoi:DJ-unif}, the exact output satisfies
\begin{align}\label{eq:qoi:baseline}
\mathcal J(u)\in \left [\mathcal J(\tilde u_h) - L_{\mathcal J}(\rho)\rho,\ \ \mathcal J(\tilde u_h)+ L_{\mathcal J}(\rho)\rho \right].
\end{align}
\end{proposition}

\begin{proof}
Let $e:=u-\tilde u_h$. Because $u \in B_{\rho}$, we have $\|e\|_V \leq \rho$.
By the mean value theorem in Banach spaces,
\begin{align*}
\displaystyle
\mathcal J(u) - \mathcal J(\tilde u_h) = \int_0^1 D \mathcal J(\tilde u_h+t e) e dt.
\end{align*}
Taking absolute values and using \eqref{eq:qoi:DJ-unif} yields
\begin{align*}
\displaystyle
|\mathcal J(u)- \mathcal J(\tilde u_h)| \leq L_{\mathcal J}(\rho)\|e\|_V \leq L_{\mathcal J}(\rho) \rho,
\end{align*}
which gives \eqref{eq:qoi:baseline}.
\end{proof}

\subsection{Adjoint-enhanced enclosure driven by the certified ingredients (C1)--(C3)}\label{subsec:qoi:adjoint}
\begin{assumption}[Lipschitz bound for the QoI derivative on $B_{\rho}$] \label{ass:qoi:DJ-lip}

The functional $\mathcal J$ is Fr\'echet differentiable on $B_{\rho}$, and there exists a computable constant $M_{\mathcal J}(\rho)\geq 0$ such that
\begin{align}\label{eq:qoi:DJ-lip}
\|D \mathcal J(w)-D \mathcal J(z)\|_{\mathcal L(V,\mathbb R)}
\leq M_{\mathcal J}(\rho)\|w-z\|_V
\quad \forall w,z\in B_{\rho}.
\end{align}
\end{assumption}
Recall the linearization $\mathcal L_{\tilde u_h}:= D \mathcal F(\tilde u_h)\in\mathcal L(V,V^*)$ and the bilinear form $B_{\tilde u_h}$ from \eqref{eq:lin-bilinear}:
\begin{align*}
\displaystyle
B_{\tilde u_h}(\delta u,v) = \langle \mathcal L_{\tilde u_h} \delta u,  v\rangle \quad \forall \delta u,v\in V.
\end{align*}
In the symmetric elliptic setting, the coercivity estimate \eqref{eq:C2-coerc} holds with a
certified constant $\alpha>0$:
\begin{align*}
B_{\tilde u_h}(v,v) \geq \alpha\|v\|_V^2 \quad \forall v\in V.
\end{align*}
Let $j_{\tilde u_h}\in V^*$ be the linear functional defined by $j_{\tilde u_h}(v):= D \mathcal J(\tilde u_h)v$. The adjoint state $z\in V$ is defined as the unique solution of
\begin{align}\label{eq:qoi:adjoint}
B_{\tilde u_h}(v,z)= j_{\tilde u_h}(v)\quad \forall v\in V.
\end{align}
Let $z_h\in V$ be any computable approximation of $z \in V$. We define the adjoint residual $\mathcal G(z_h)\in V^*$ as
\begin{align}\label{eq:qoi:adj-res}
\langle \mathcal G(z_h),v\rangle := j_{\tilde u_h}(v)-B_{\tilde u_h}(v,z_h) \quad \forall v\in V.
\end{align}

\begin{lemma}[Adjoint error controlled by the adjoint residual]\label{lem:qoi:adj-err}
Let $z \in V$ solve \eqref{eq:qoi:adjoint} and let $z_h\in V$ be arbitrary. Then,
\begin{align}\label{eq:qoi:adj-err}
\|z-z_h\|_V \leq \alpha^{-1} \|\mathcal G(z_h)\|_{V^*}.
\end{align}
\end{lemma}

\begin{proof}
By \eqref{eq:qoi:adj-res} and \eqref{eq:qoi:adjoint}, for any $v\in V$,
$\langle\mathcal G(z_h),v\rangle = B_{\tilde u_h}(v,z-z_h)$. Therefore,
\begin{align*}
\displaystyle
\|\mathcal G(z_h)\|_{V^*}
=\sup_{0\neq v\in V}\frac{B_{\tilde u_h}(v,z-z_h)}{\|v\|_V}
\geq \alpha \|z-z_h\|_V,
\end{align*}
which yields \eqref{eq:qoi:adj-err}.
\end{proof}

\begin{theorem}[Certified adjoint-enhanced enclosure]\label{thm:qoi:adjoint}
Assume (C1)--(C3), and fix $0<\rho\leq\bar\rho$ satisfying $p(\rho)\leq 0$ and $q(\rho)<1$ in Theorem~\ref{thm:NK}. Let $u\in B_{\rho}$ be the solution certified by that theorem, and suppose that $\mathcal J$ satisfies Assumption~\ref{ass:qoi:DJ-lip}. Let $z_h\in V$ be any computable approximation of the adjoint state in \eqref{eq:qoi:adjoint}, and let $\mathcal G(z_h)$ be its adjoint residual \eqref{eq:qoi:adj-res}. Then, the exact output admits the certified enclosure
\begin{align} \label{eq:qoi:adjoint-enclosure}
\mathcal J(u) \in \left [ \mathcal J(\tilde u_h) - E_{\mathcal J}(\rho;z_h),\  \mathcal J(\tilde u_h) + E_{\mathcal J}(\rho;z_h) \right],
\end{align}
where
\begin{align}\label{eq:qoi:EJ}
E_{\mathcal J}(\rho;z_h)
&:=
\left|\langle \mathcal F(\tilde u_h), z_h\rangle \right|
+\frac12 L(\rho)\rho^2 \|z_h\|_V 
+\left(\mathfrak r+\frac12 L(\rho) \rho^2\right) \alpha^{-1}\|\mathcal G(z_h)\|_{V^*}
+\frac12 M_{\mathcal J}(\rho)\rho^2.
\end{align}
\end{theorem}

\begin{proof}
Let $e:=u-\tilde u_h$. Then, $\|e\|_V \leq \rho$ and $\mathcal F(u)=0$. We write
\begin{align*}
\displaystyle
\mathcal J(u) - \mathcal J(\tilde u_h)=D \mathcal J(\tilde u_h) e+\mathcal R_{\mathcal J},
\quad
\mathcal R_{\mathcal J} := \mathcal J(u) - \mathcal J(\tilde u_h) - D \mathcal J(\tilde u_h)e.
\end{align*}
From \eqref{eq:qoi:DJ-lip},
\begin{align*}
\displaystyle
|\mathcal R_{\mathcal J}|
&\leq \int_0^1 \left \| D \mathcal J (\tilde u_h + t e) - D \mathcal J(\tilde u_h) \right \|_{V^*} \| e \|_V dt \\
&\leq \frac{1}{2} M_{\mathcal J}(\rho) \| e \|_V^2 \leq \frac{1}{2} M_{\mathcal J}(\rho) \rho^2.
\end{align*}
Because $D \mathcal J(\tilde u_h)e=j_{\tilde u_h}(e)$ and $z$ solves \eqref{eq:qoi:adjoint}, we have
\begin{align*}
\displaystyle
j_{\tilde u_h}(e)=B_{\tilde u_h}(e,z)=\langle \mathcal L_{\tilde u_h}e,z\rangle = \langle \mathcal L_{\tilde u_h}e,z_h\rangle +\langle \mathcal L_{\tilde u_h}e, z-z_h\rangle.
\end{align*}
Using $\mathcal F(u)=0$ and $\mathcal L_{\tilde u_h}=D\mathcal F(\tilde u_h)$, we set
\begin{align*}
\displaystyle
\mathcal R_{\mathcal F}:=\mathcal F(u)-\mathcal F(\tilde u_h)-\mathcal L_{\tilde u_h}e,
\quad \text{so that}\quad
\mathcal L_{\tilde u_h}e=-\mathcal F(\tilde u_h)-\mathcal R_{\mathcal F}.
\end{align*}
From (C3) and $\|e\|_V \leq \rho$, the standard remainder estimate yields
\begin{align*}
\displaystyle
\| \mathcal R_{\mathcal F} \|_{V^*}
&\leq \frac{1}{2} L(\rho)\|e\|_V^2 \leq \frac{1}{2}L(\rho)\rho^2.
\end{align*}
Therefore,
\begin{align*}
\displaystyle
\left |\langle \mathcal L_{\tilde u_h}e,z_h\rangle \right|
=\left |\langle \mathcal F(\tilde u_h),z_h\rangle+\langle \mathcal R_{\mathcal F},z_h\rangle \right|
\leq \left |\langle \mathcal F(\tilde u_h),z_h\rangle \right|
+\frac12 L(\rho)\rho^2 \|z_h\|_V.
\end{align*}
Furthermore,
\begin{align*}
\displaystyle
\left |\langle \mathcal L_{\tilde u_h}e, z-z_h\rangle \right|
\leq \|\mathcal L_{\tilde u_h}e\|_{V^*}\,\|z-z_h\|_V
\leq \left(\|\mathcal F(\tilde u_h)\|_{V^*}+\|\mathcal R_{\mathcal F}\|_{V^*} \right) \|z-z_h\|_V.
\end{align*}
Using (C1), i.e.,  $\|\mathcal F(\tilde u_h)\|_{V^*} \leq \mathfrak r$, $\|\mathcal R_{\mathcal F}\|_{V^*}\le\tfrac12 L(\rho)\rho^2$, and Lemma~\ref{lem:qoi:adj-err} gives
\begin{align*}
\displaystyle
\left |\langle \mathcal L_{\tilde u_h}e, z-z_h\rangle \right|
\leq \left(\mathfrak r + \frac{1}{2} L(\rho)\rho^2\right)\alpha^{-1}\|\mathcal G(z_h)\|_{V^*}.
\end{align*}
Combining the above estimates with the bound for $\mathcal R_{\mathcal J}$ yields $|\mathcal J(u) - \mathcal J(\tilde u_h)|\le E_{\mathcal J}(\rho;z_h)$, which implies \eqref{eq:qoi:adjoint-enclosure}.
\end{proof}

\begin{remark}[Relation to earlier output bounds]\label{rem:qoi:linear-literature}
The adjoint argument in Theorem~\ref{thm:qoi:adjoint} is familiar from guaranteed output estimation for linear elliptic problems \citep{AinsworthRankin2012QoI,MallikVohralikYousef2020QoI,ParesNguyenDiezPeraire2021QoI}. Here, it is applied only after Theorem~\ref{thm:NK} has located a semilinear solution in $B_\rho$. With $e=u-\tilde u_h$, the bound $\|e\|_V\leq\rho$ gives
\begingroup\begin{align*}
\|\mathcal F(u)-\mathcal F(\tilde u_h)-D\mathcal F(\tilde u_h)e\|_{V^*}
&\leq \tfrac12 L(\rho)\rho^2,\\
|\mathcal J(u)-\mathcal J(\tilde u_h)-D\mathcal J(\tilde u_h)e|
&\leq \tfrac12 M_{\mathcal J}(\rho)\rho^2.
\end{align*}\endgroup
Both bounds appear in \eqref{eq:qoi:EJ}, together with the primal residual pairing and the adjoint residual majorant. The latter is obtained by the flux procedure described below. The first remainder vanishes for a linear equation, and the second for a linear output. The adjoint-based enclosure need not improve on Proposition~\ref{prop:qoi:baseline} for any $z_h$. Its width depends on the computed adjoint and the bounds used for its residual.
\end{remark}

\begin{remark}[Making $\|\mathcal G(z_h)\|_{V^\ast}$ fully computable]\label{rem:adjoint-residual-computable}
The scalar term $\langle \mathcal F(\tilde u_h), z_h\rangle$ in \eqref{eq:qoi:EJ} is directly computable once an approximation $z_h\in V$ is available. (If strict rigor is required, the evaluation of this duality pairing can be carried out with outward rounding/interval arithmetic at the scalar level.)

It remains to obtain a guaranteed upper bound for the dual norm $\|\mathcal G(z_h)\|_{V^\ast}$, where $\mathcal G(z_h)=j_{\tilde u_h}- \mathcal L_{\tilde u_h}z_h\in V^\ast$ is defined by \eqref{eq:qoi:adj-res}. Because we work in the symmetric elliptic setting, $\mathcal L_{\tilde u_h}$ is the operator induced by the bilinear form $B_{\tilde u_h}$, i.e.,
\begin{align*}
\displaystyle
\langle \mathcal L_{\tilde u_h} w, v\rangle := B_{\tilde u_h}(w,v)\quad \forall\, w,v\in V .
\end{align*}
Therefore, the exact adjoint state $z\in V$ solves the linear diffusion--reaction problem
\begin{align*}
\displaystyle
\mathcal L_{\tilde u_h} z = j_{\tilde u_h}\quad \text{in } V^\ast ,
\end{align*}
and $\mathcal G(z_h)$ is precisely the residual of this linear problem at $z_h$. For the volume-type QoIs in Section~\ref{subsec:qoi:examples}, the functional $j_{\tilde u_h}$ admits an $L^2$-density $\psi\in L^2(\Omega)$ such that
\begin{align*}
\displaystyle
j_{\tilde u_h}(v)=\int_\Omega \psi v dx \quad \forall v\in V.
\end{align*}
For instance, $\psi$ is the given weight for the linear output, and $\psi=\tilde u_h$ for the quadratic
$L^2$-energy output.

To bound $\|\mathcal G(z_h)\|_{V^\ast}$, we apply the same equilibrated-flux residual certification as in Section~\ref{sec:flux}, to the linear adjoint operator. Concretely, in the semilinear diffusion--reaction case \eqref{eq:F-semi}, the adjoint residual reads
\begin{align*}
\displaystyle
\langle \mathcal G(z_h), v\rangle
= \int_\Omega \psi  v dx
- \int_\Omega A\nabla z_h\cdot\nabla v dx
- \int_\Omega \partial_s c(x,\tilde u_h)  z_h  v dx
\quad \forall v\in V,
\end{align*}
where $\psi$ is the $L^2(\Omega)$-density of $j_{\tilde u_h}$ (e.g., $\psi$ is the given weight for linear outputs, and $\psi=\tilde u_h$ for the quadratic $L^2$-energy output in Section~\ref{subsec:qoi:examples}).

To apply the $L^2$ residual estimate, we assume that the adjoint source $\partial_s c(\cdot,\tilde u_h)z_h-\psi$ belongs to $L^2(\Omega)$. This holds in the finite-element cubic example of Section~\ref{sec:numerics}. Given $z_h$, we define the adjoint source term as
\begin{align*}
\displaystyle
s_h^{\ast} :=  \partial_s c(\cdot,\tilde u_h) z_h - \psi,
\quad
r_h^{\ast}|_T := \Pi_T^0(s_h^{\ast}) \quad \forall T\in\mathbb T_h.
\end{align*}
Let $\sigma_h^{\ast}\in H(\mathrm{div};\Omega)$ be any flux satisfying the (adjoint) divergence target $\mathrm{div} \sigma_h^{\ast}=r_h^{\ast}$ elementwise. Then, the element residual is
\begin{align*}
\displaystyle
R_T^{\ast}(\sigma_h^{\ast})
:= s_h^{\ast} - \mathrm{div}\sigma_h^{\ast}\quad \text{in }T.
\end{align*}
Because $\|\mathcal G(z_h)\|_{V^\ast} = \| - \mathcal G(z_h)\|_{V^\ast}$,  we apply Lemma~\ref{thm:res-majorant} to $- \mathcal G(z_h)= \mathcal L_{\tilde u_h}z_h - j_{\tilde u_h}$. Then,
\begin{align*}
\displaystyle
\|\mathcal G(z_h)\|_{V^\ast} = \| - \mathcal G(z_h)\|_{V^\ast}
\leq \|\sigma_h^{\ast}-A\nabla z_h\|_{A^{-1}}
+
\alpha_0^{- \frac{1}{2}}\left(\sum_{T \in \mathbb T_h} \left( \frac{h_T}{\pi} \right)^2 
\|R_T^{\ast}(\sigma_h^{\ast})\|_{L^2(T)}^2 \right)^{\frac{1}{2}}.
\end{align*}
The flux $\sigma_h^{\ast}$ can be constructed by the same explicit Marini-type route of Section~\ref{sec:fluxrecon} (with $s_h$ replaced by $s_h^{\ast}$), so the computation of $\|\mathcal G(z_h)\|_{V^\ast}$ requires only the same low-order postprocessing as for $\|\mathcal F(\tilde u_h)\|_{V^\ast}$.
\end{remark}

\subsection{Examples of volume-type QoIs and admissible constants}\label{subsec:qoi:examples}
We specialize the abstract output bounds to two volume-type QoIs whose derivatives admit $L^2(\Omega)$-densities. Throughout, we use the continuous embedding
\begin{align}\label{eq:embed-L2}
\|v\|_{L^2(\Omega)}\le C_2 \|v\|_V \quad \forall v\in V,
\end{align}
with $C_2>0$ as introduced in \eqref{eq:CP}.

\paragraph*{Linear output functional.}
Let $\mathcal J:V\to\mathbb R$ be given by
\begin{align*}
\displaystyle
\mathcal J(v):=\int_\Omega \psi v dx, \quad \psi\in L^2(\Omega).
\end{align*}
Then, $\mathcal J$ is linear and its Fr\'echet derivative is independent of the evaluation point:
for every $w\in V$ and every direction $\delta\in V$,
\begin{align*}
\displaystyle
D \mathcal J(w) \delta=\int_\Omega \psi \delta dx.
\end{align*}
Consequently, $D \mathcal J(w)-D \mathcal J(z)\equiv 0$ for all $w,z\in V$, so the Lipschitz constant in \eqref{eq:qoi:DJ-lip} can be chosen as
\begin{align*}
\displaystyle
M_{\mathcal J}(\rho)=0.
\end{align*}
Furthermore, the right-hand side functional in the adjoint problem is precisely $j_{\tilde u_h}=D \mathcal J(\tilde u_h)$, and its dual norm satisfies
\begin{align*}
\displaystyle
\|j_{\tilde u_h}\|_{V^*}
=\sup_{0\neq v\in V}\frac{\left |\int_\Omega \psi v dx \right|}{\|v\|_V}
\leq \sup_{0\neq v\in V}\frac{\|\psi\|_{L^2(\Omega)} \|v\|_{L^2(\Omega)}}{\|v\|_V}
\leq \|\psi\|_{L^2(\Omega)}C_2,
\end{align*}
where we used Cauchy--Schwarz in $L^2(\Omega)$ and the embedding \eqref{eq:embed-L2}. 

For the baseline enclosure in Proposition~\ref{prop:qoi:baseline}, one may take $L_{\mathcal J}(\rho) := \|\psi\|_{L^2(\Omega)} C_2$.

\paragraph*{Quadratic $L^2$-energy output.}
Let $\mathcal J:V\to\mathbb R$ be given by
\begin{align*}
\displaystyle
\mathcal J(v):=\frac12\|v\|_{L^2(\Omega)}^2=\frac12\int_\Omega v^2 dx.
\end{align*}
A direct computation shows that for every $w\in V$ and $\delta\in V$,
\begin{align*}
\displaystyle
D \mathcal J(w) \delta = \int_\Omega w \delta dx,
\end{align*}
that is, $D \mathcal J(w)\in V^*$ is the $L^2$-pairing with $w$. Therefore, for any $w,z\in V$,
\begin{align*}
\displaystyle
(D \mathcal J(w) - D \mathcal J(z))\delta = \int_\Omega (w-z) \delta dx.
\end{align*}
Using Cauchy--Schwarz in $L^2(\Omega)$ and the embedding \eqref{eq:embed-L2}, we obtain
\begin{align*}
\displaystyle
\|D \mathcal J(w) - D \mathcal J(z)\|_{\mathcal L(V,\mathbb R)}
&=\sup_{0\neq \delta\in V}\frac{\left |\int_\Omega (w-z) \delta dx \right|}{\|\delta\|_V} \\
&\leq \sup_{0\neq \delta\in V}\frac{\|w-z\|_{L^2(\Omega)} \|\delta\|_{L^2(\Omega)}}{\|\delta\|_V} \leq C_2 \|w-z\|_{L^2(\Omega)}.
\end{align*}
Applying \eqref{eq:embed-L2} yields
\begin{align*}
\displaystyle
\|D \mathcal J(w)- D \mathcal J(z)\|_{\mathcal L(V,\mathbb R)}
\leq C_2^2\,\|w-z\|_V.
\end{align*}
Therefore, \eqref{eq:qoi:DJ-lip} holds with the admissible choice
\begin{align*}
\displaystyle
M_{\mathcal J}(\rho)=C_2^2.
\end{align*}
Because $j_{\tilde u_h}=D \mathcal J(\tilde u_h)$ and $D \mathcal J(\tilde u_h) v = \int_\Omega \tilde u_h v dx$, we similarly have
\begin{align*}
\displaystyle
\|j_{\tilde u_h}\|_{V^*}
=\sup_{0\neq v\in V}\frac{\left |\int_\Omega \tilde u_h v dx \right|}{\|v\|_V}
\leq C_2\|\tilde u_h\|_{L^2(\Omega)}.
\end{align*}
For Proposition~\ref{prop:qoi:baseline}, we use the derivative bound
\begingroup
\begin{align*}
\displaystyle
\sup_{w\in B_\rho}\|D\mathcal J(w)\|_{V^*}\le C_2\sup_{w\in B_\rho}\|w\|_{L^2(\Omega)} \leq C_2 \left(\|\tilde u_h\|_{L^2(\Omega)}+C_2\rho \right).
\end{align*}
\endgroup

\begin{remark}[Baseline vs.\ adjoint-based output bounds]
Proposition~\ref{prop:qoi:baseline} provides a baseline enclosure of $\mathcal J(u)$ that does not require solving the adjoint problem. Its bound depends on the worst-case sensitivity constant $\displaystyle L_{\mathcal J}(\rho):=\sup_{w\in B_{\rho}}\|D \mathcal J(w)\|_{V^*}$. For the quadratic $L^2$-energy output, one may choose
\begin{align*}
\displaystyle
L_{\mathcal J}(\rho)= C_2\bigl(\|\tilde u_h\|_{L^2(\Omega)}+C_2\rho\bigr),
\end{align*} 
so that Proposition~\ref{prop:qoi:baseline} is fully explicit. In contrast, once an adjoint approximation is available, Theorem~\ref{thm:qoi:adjoint} replaces this worst-case sensitivity by goal-oriented residual terms (e.g., $\langle \mathcal F(\tilde u_h),z_h\rangle$ and $\|\mathcal G(z_h)\|_{V^*}$), which typically yields a much sharper certified bound in practice.
\end{remark}

\FloatBarrier
\section{Numerical experiments}\label{sec:numerics}
We demonstrate the certification procedure for the monotone semilinear model considered below. The numerical experiments give the computed NK components, the verification radius, and the resulting output enclosures. Sections~\ref{subsec:numerics:model}--\ref{subsec:numerics:qoi} use uniform right-triangular meshes, whereas Section~\ref{subsec:anisotropic-numerics} repeats the same certification workflow on anisotropic right-triangular meshes.

The bounds established above are rigorous analytical statements. The numerical values below are obtained by evaluating the corresponding quantities in ordinary floating-point arithmetic. In particular, the tables should not be read as interval-arithmetic computer-assisted proofs; see Remark~\ref{rem:sec9-floating-point}.

\begin{remark}[Implementation]

All computations in Section~\ref{sec:numerics} were performed with a pure Python/NumPy/SciPy implementation of the same low-order $\mathbb{P}^1$--CR--RT certification workflow. The code explicitly assembles the conforming $P^1$ nonlinear problem, the CR auxiliary problem, and the Marini-type $\mathbb{RT}^0$ flux reconstruction. The only change between the uniform and anisotropic tests is the mesh family: Sections~\ref{subsec:numerics:model}--\ref{subsec:numerics:qoi} use $N_x=N_y=N$, while Section~\ref{subsec:anisotropic-numerics} uses $N_x=M$ and $N_y=M^2$.

\end{remark}

\begin{remark}[Floating-point qualification]
\label{rem:sec9-floating-point}

All numerical quantities presented in Section~\ref{sec:numerics}, including $\mathfrak r$, $L(\rho)$, $p(\rho)$, $q(\rho)$, and the output bounds, were evaluated in standard double precision and without outward rounding. A fully verified implementation would also have to control rounding errors in the discrete solves, the flux reconstruction and equilibration, the norm and integral evaluations, and the final Newton--Kantorovich tests, for example by verified linear algebra and interval arithmetic. We do not carry out that additional implementation step here.

\end{remark}

\subsection{Model problem}\label{subsec:numerics:model}
We consider the semilinear diffusion--reaction problem as a special case of the semilinear model in Section~\ref{sec:semilinear-problem} with $A(x)\equiv I$ and $c(x,u)=u^3$:
\begin{align}\label{eq:numerics:pde}
-\varDelta u + u^3 = f  \quad  \text{in }\Omega, \quad u = 0 \quad \text{on } \partial\Omega,
\end{align}
where $\Omega \subset \mathbb R^2$. To enable an accuracy check, we employ a manufactured exact solution $u^*$, and we define $f := -\varDelta u^* + (u^*)^3$. We set
\begin{align*}
\displaystyle
u^*(x,y)=\sin(\pi x)\sin(\pi y)\quad\text{on }\Omega=(0,1)^2, \quad f = 2\pi^2 u^* + (u^*)^3.
\end{align*}
We employ conforming finite elements $V_h\subset V=H_0^1(\Omega)$ as in Section~2.5. We use continuous piecewise polynomials of degree $k=1$ on a family of conforming simplicial meshes $\{\mathbb T_h\}$. The discrete nonlinear problem is solved by a standard Newton method, and we denote by $\tilde u_h \in V_h$ the computed approximation (typically the final Newton iterate). Given an iterate $u_h^k\in V_h$, compute the Newton increment $\delta u_h^k\in V_h$ by
\begin{align*}
\displaystyle
\langle D \mathcal F(u_h^k)\,\delta u_h^k, v_h\rangle = - \langle \mathcal F(u_h^k), v_h\rangle
\quad \forall v_h\in V_h,
\end{align*}
and update $u_h^{k+1}=u_h^k+\delta u_h^k$. We stop if $\|\delta u_h^k\|_V \le \texttt{tolNewton}$ (or if $k$ reaches $\texttt{maxNewtonIters}$) with $\texttt{tolNewton}=10^{-12}$ and $\texttt{maxNewtonIters}=25$ in the computations below. (This is not a certified criterion; certification is performed afterwards by the bounds in Sections~4--8.) In the semilinear setting, we take the reference tensor $A=I$, so that the fixed energy norm \eqref{eq:V-norm} reduces to
\begin{align*}
\displaystyle
\|v\|_V = \left( \int_\Omega |\nabla v|^2 dx \right)^{\frac{1}{2}}.
\end{align*}
Therefore, $\alpha_0=1$ in \eqref{eq:A0-coerc}, and the embedding constants $C_p$ in \eqref{eq:CP} are those for $H_0^1(\Omega)\hookrightarrow L^p(\Omega)$ in this metric. The reaction $c(u)=u^3$ is monotone, hence the stability constant $\alpha$ can be certified by the coercivity argument of Section \ref{subsec:monotone-bound}. We use this model to test the residual bound (Sections~\ref{sec:flux}--\ref{sec:fluxrecon}), the Lipschitz estimate (Section~\ref{sec:C3}), and the radius selection (Algorithm~\ref{alg:rho}). Strong monotonicity also gives continuous well-posedness without an NK argument. Our numerical test concerns the size of the computable local ball and the output bounds, rather than the need for an NK existence theorem for this example.

\begin{figure}[!htbp]
\centering
\includegraphics[width=0.45\textwidth]{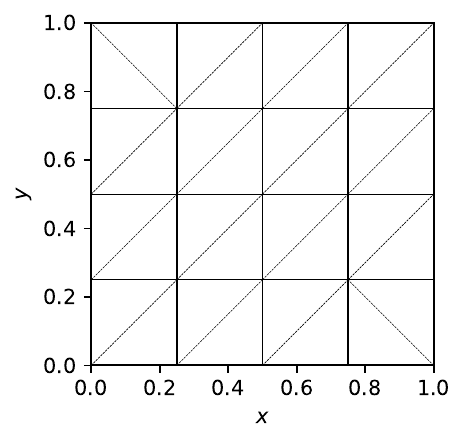}
\caption{Standard right-triangular mesh with  $N=4$.}
\label{fig:standard-mesh}
\end{figure}

\subsection{Computed NK inputs (C1)--(C3)}\label{subsec:numerics:nk-inputs}
We detail how the three NK inputs in Section \ref{subsec:NK-assumptions} are instantiated for \eqref{eq:numerics:pde}.

\paragraph*{C1.}
We compute a flux $\sigma_h\in H(\mathrm{div};\Omega)$ satisfying the mean equilibration condition \eqref{eq:mean-equil} by the explicit Marini-type construction of Section \ref{sec:fluxrecon}. We then evaluate the two terms \eqref{eq:eta_mis}--\eqref{eq:eta_osc}:
\begin{align*}
\displaystyle
\eta_{\mathrm{mis}}(\sigma_h)
&= \|\sigma_h-\nabla\tilde u_h\|_{L^2(\Omega)^d}, \\
\eta_{\mathrm{osc}}(\sigma_h)
&=
\left (\sum_{T\in\mathbb T_h} \left (\frac{h_T}{\pi} \right)^2\| \mathcal R_T(\sigma_h)\|_{L^2(T)}^2 \right)^{\frac{1}{2}}.
\end{align*}
The residual quantity corresponding to \eqref{eq:C1} is set to
\begin{align*}
\displaystyle
\mathfrak r := \eta(\sigma_h) = \eta_{\mathrm{mis}}(\sigma_h)+\eta_{\mathrm{osc}}(\sigma_h),
\end{align*}
see Lemma~\ref{thm:res-majorant}.

\paragraph*{C2.}
For $c(u)=u^3$, we have $\partial_s c(u)=3u^2\ge 0$ pointwise. Therefore, the monotonicity condition in Section \ref{subsec:monotone-bound} holds with $A=I$ and $\alpha_0=1$, the coercivity estimate \eqref{eq:C2-coerc} is satisfied with
\begin{align*}
\displaystyle
\alpha := 1,
\end{align*}
so that the stability ingredient \eqref{eq:C2} holds with $\|\mathcal L_{\tilde u_h}^{-1}\|_{L(V^\ast,V)}\leq 1$.

\paragraph*{C3.}
We use the semilinear route of Section \ref{subsec:C3-semi-general}. Because $\partial_s c(u)=3u^2$,
 for any $w,z\in B_\rho$,
\begin{align*}
\displaystyle
\partial_s c(w)-\partial_s c(z) = 3(w+z)(w-z).
\end{align*}
We choose $p=4$ in Lemma \ref{lem:C3-semi} (hence $q=2$) and employ the embedding \eqref{eq:CP} with constant $C_4$. Proceeding as in Example \ref{ex:AC-reaction} (which treats the same algebraic structure) with $\kappa\equiv 1$, we obtain \eqref{eq:C3-semi-gamma} with
\begin{align*}
\displaystyle
\Gamma(\rho)=6 C_4^2 \left(\|\tilde u_h\|_V+\rho \right),
\end{align*}
and thus the Lipschitz function in \eqref{eq:C3} can be taken as
\begin{align}\label{eq:numerics:Lrho}
L(\rho)=C_4^2 \Gamma(\rho)
=6 C_4^4 \left(\|\tilde u_h\|_V+\rho \right).
\end{align}
The constant $C_4$ is treated as a domain-dependent input (Remark \ref{rem=embed}). For the unit square $\Omega=(0,1)^2$, we take $C_4$ from the verified numerical inclusion discussed in \cite[Theorem 1.1]{TanSekMizOis17}:
\begin{align*}
\displaystyle
C_4(\Omega)\in[ 0.28524446071925, 0.28524446071929 ].
\end{align*}
In the computations, we fix the conservative choice
\begin{align*}
\displaystyle
C_4 := 0.28524446071929.
\end{align*}

\subsection{Computed selection of the verification radius}\label{subsec:rho-selection}
With $\mathfrak r$, $\alpha$, and $L(\rho)$ at hand, we set
\begin{align*}
\displaystyle
\eta = \frac{\mathfrak r}{\alpha},
\end{align*}
and compute a verification radius $\rho$ by the simple one-dimensional procedure in Algorithm~\ref{alg:rho}. We call $\rho>0$ {admissible} if the NK conditions \eqref{eq:NK-conds} hold, i.e., $p(\rho)\leq 0$ and $q(\rho)<1$, where $p(\rho),q(\rho)$ are defined in \eqref{eq:pq}, see Table \ref{tab:sec9-cert}. In the affine Lipschitz model $L(\rho)=L_0+L_1\rho$ with $L_0\ge 0$ and $L_1>0$, the admissible set $\{0<\rho\leq\bar\rho:\ p(\rho) \leq 0,\ q(\rho)<1\}$ is either empty, a singleton, or an interval. For $\eta>0$, a nonempty admissible interval has a positive left endpoint; for $\eta=0$, it may extend down to zero. Therefore, once a bracket is found, bisection is well-defined and, in exact arithmetic, returns an admissible radius. In the computations below, however, this radius is evaluated in ordinary double precision, as noted in Remark~\ref{rem:sec9-floating-point}. The proofs and implementation details are deferred to Appendix~\ref{app:rho-selection}.

\begin{remark}
In our implementation, we work in a normalized setting with $\alpha=1$, so that $\eta=\mathfrak r$.
\end{remark}

\begin{table}[htb]
\centering
\begingroup
\footnotesize
\setlength{\tabcolsep}{3pt}
\sisetup{
  scientific-notation = true,
  round-mode = figures,
  round-precision = 3
}
\caption{Computed NK data on uniform right-triangular meshes. The table gives the scaled residual parameter $\eta=\mathfrak r/\alpha$, the selected radius $\rho$, the admissibility indicators $q(\rho)$ and $p(\rho)$, and the ratio $\rho/\eta$.}
\label{tab:sec9-cert}
\begin{adjustbox}{max width=\linewidth}
\begin{tabular}{r S S S S S S}
\toprule
{$N$} & {$h_{\max}$} & {$\eta$} & {$\rho$} & {$q(\rho)$} & {$p(\rho)$} & {$\rho/\eta$} \\
\midrule
16 & 8.838835e-02 & 2.6903287414258586e-01 & 5.3806574828517173e-01 & 5.8784730177164825e-02 & -2.5321784922732682e-01 & 2.000000e+00 \\
32 & 4.419417e-02 & 1.3034010971039087e-01 & 2.6068021942078173e-01 & 2.5677710390283864e-02 & -1.2699327412100964e-01 & 2.000000e+00 \\
64 & 2.209709e-02 & 6.4078361326617914e-02 & 1.2815672265323583e-01 & 1.1957787836111627e-02 & -6.3312125876988520e-02 & 2.000000e+00 \\
128 & 1.104854e-02 & 3.1760667939535241e-02 & 6.3521335879070481e-02 & 5.7649040330984558e-03 & -3.1577570736836716e-02 & 2.000000e+00 \\
256 & 5.524272e-03 & 1.5810020887994110e-02 & 3.1620041775988220e-02 & 2.8297548984445049e-03 & -1.5765282403941799e-02 & 2.000000e+00 \\
\bottomrule
\end{tabular}
\end{adjustbox}
\endgroup
\end{table}

\begin{table}[htb]
\centering
\begingroup
\footnotesize
\setlength{\tabcolsep}{3pt}
\sisetup{
  scientific-notation = true,
  round-mode = figures,
  round-precision = 3
}
\caption{Sanity check against the manufactured solution. The table gives the true energy error $\|u^*-\tilde u_h\|_V$, the selected radius $\rho$, and their ratio. The column ``inside?'' indicates whether the manufactured solution lies inside the corresponding ball $B(\tilde u_h,\rho)$, i.e., $\|u^*-\tilde u_h\|_V\le \rho$.}
\label{tab:sec9-sanity}
\begin{adjustbox}{max width=\linewidth}
\begin{tabular}{r S S S S c}
\toprule
{$N$} & {$h_{\max}$} & {$\|u^*-\tilde u_h\|_V$} & {$\rho$} & {$\|u^*-\tilde u_h\|_V/\rho$} & {inside?} \\
\midrule
16 & 8.838835e-02 & 2.1750010067486833e-01 & 5.3806574828517173e-01 & 4.042259e-01 & yes \\
32 & 4.419417e-02 & 1.0897490363863399e-01 & 2.6068021942078173e-01 & 4.180406e-01 & yes \\
64 & 2.209709e-02 & 5.4513767537551147e-02 & 1.2815672265323583e-01 & 4.253680e-01 & yes \\
128 & 1.104854e-02 & 2.7260115992612896e-02 & 6.3521335879070481e-02 & 4.291490e-01 & yes \\
256 & 5.524272e-03 & 1.3630460224232775e-02 & 3.1620041775988220e-02 & 4.310703e-01 & yes \\
\bottomrule
\end{tabular}
\end{adjustbox}
\endgroup
\end{table}

\subsection{Computed output enclosures: baseline and adjoint-enhanced bounds}\label{subsec:numerics:qoi}
We consider the two volume-type QoIs from Section~\ref{subsec:qoi:examples}.
For $\Omega=(0,1)^2$ with $A=I$, we use the explicit Poincar\'e constant $C_2=1/(\pi\sqrt{2})$ in \eqref{eq:embed-L2}, obtained from the first Dirichlet eigenvalue $2\pi^2$ of the square.
The quantities $W_{\rm base}$ and $W_{\rm adj}$ reported below are interval half-widths, not full widths. For the linear output, $W_{\rm base}=C_2\rho$. For the quadratic output, the tables use the conservative value $W_{\rm base}=C_2\|\tilde u_h\|_{L^2(\Omega)}\rho+\tfrac32 C_2^2\rho^2$, rather than the smaller sufficient half-width obtained with $C_2^2\rho^2$ in place of $\tfrac32 C_2^2\rho^2$ in Proposition~\ref{prop:qoi:baseline}. The ratio \textnormal{err/adj} is the true output error divided by $W_{\rm adj}$.

\paragraph*{QoI 1: Linear output.}
Let $\mathcal J_1(u)=\int_\Omega \psi u\,dx$ with $\psi\equiv 1$ (domain average up to scaling). Then, $M_{\mathcal J_1}(\rho)=0$ and the baseline constant $L_{\mathcal J_1}(\rho)$ in Assumption \ref{ass:qoi:DJ-unif} can be taken as $L_{\mathcal J_1}(\rho)= \|j_{\tilde u_h}\|_{V^\ast}\le C_2\|\psi\|_{L^2(\Omega)}$, see Section \ref{subsec:qoi:examples}.

\paragraph*{QoI 2: Quadratic $L^2$-energy.}
Let $\mathcal J_2(u)=\frac12\|u\|_{L^2(\Omega)}^2$. Then, $M_{\mathcal J_2}(\rho)$ can be chosen as $M_{\mathcal J_2}(\rho)=C_2^2$, see Section \ref{subsec:qoi:examples}.

\paragraph*{Baseline enclosures (Proposition \ref{prop:qoi:baseline})}\label{subsec:numerics:qoi-baseline}
For each QoI $\mathcal J_i$, $i=1,2$, Proposition \ref{prop:qoi:baseline} yields
\begin{align*}
\displaystyle
\mathcal J_i(u)\in [\mathcal J_i(\tilde u_h)- L_{ \mathcal  J_i}(\rho)\rho, \mathcal  J_i(\tilde u_h)+ L_{\mathcal  J_i}(\rho)\rho].
\end{align*}
\paragraph*{Adjoint-enhanced enclosures (Theorem \ref{thm:qoi:adjoint})}\label{subsec:numerics:qoi-adjoint}
We compute an approximation $z_h\in V_h$ of the adjoint state $z\in V$ defined by \eqref{eq:qoi:adjoint}, i.e.,
\begin{align*}
\displaystyle
B_{\tilde u_h}(v,z)=j_{\tilde u_h}(v) \quad \forall v\in V,
\end{align*}
with $B_{\tilde u_h}$ induced by the linearization $\mathcal L_{\tilde u_h}=D \mathcal F(\tilde u_h)$. For \eqref{eq:numerics:pde}, we have
\begin{align*}
\displaystyle
B_{\tilde u_h}(w,v) = \int_\Omega \nabla w \cdot \nabla v dx + \int_\Omega 3(\tilde u_h)^2 w v dx.
\end{align*}
We then evaluate the computed error budget $E_{\mathcal J}(\rho;z_h)$ in \eqref{eq:qoi:EJ}. The remaining term $\|\mathcal G(z_h)\|_{V^\ast}$ is bounded in a guaranteed manner by applying Lemma \ref{thm:res-majorant} to the linear adjoint residual problem described in Section \ref{subsec:qoi:adjoint}.

For each QoI, Tables~\ref{tab:sec9-qoi-linear} and~\ref{tab:sec9-qoi-quadratic} give the computed value $\mathcal J_i(\tilde u_h)$, the true error for the manufactured solution, the baseline and adjoint half-widths, and the ratio of the true error to the adjoint half-width.

\begin{table}[htb]
\centering
\begingroup
\footnotesize
\setlength{\tabcolsep}{3pt}
\sisetup{
  scientific-notation = true,
  round-mode = figures,
  round-precision = 3
}
\caption{Computed enclosure for the linear quantity of interest $\mathcal J_1(u)=\int_\Omega u\,dx$ on uniform meshes. Here, $W_{\rm base}$ and $W_{\rm adj}$ denote the half-widths of the baseline and adjoint-based intervals, respectively. For the manufactured test, the table also gives the true error and the ratio err/adj.}
\label{tab:sec9-qoi-linear}
\begin{adjustbox}{max width=\linewidth}
\begin{tabular}{r S S S S S S}
\toprule
{$N$} & {$h_{\max}$} & {$\mathcal J_1(\tilde u_h)$} & {$|\mathcal J_1(u^*)-\mathcal J_1(\tilde u_h)|$} &
{$W_{\rm base}$} & {$W_{\rm adj}$} & {err/adj} \\
\midrule
16 & 8.838835e-02 & 4.0172955528567506e-01 & 3.555179e-03 & 1.211073e-01 & 9.644936e-03 & 3.686058e-01 \\
32 & 4.419417e-02 & 4.0438511617436784e-01 & 8.996184e-04 & 5.867366e-02 & 2.219335e-03 & 4.053550e-01 \\
64 & 2.209709e-02 & 4.0505915810861437e-01 & 2.255765e-04 & 2.884540e-02 & 5.315310e-04 & 4.243901e-01 \\
128 & 1.104854e-02 & 4.0522829858448589e-01 & 5.643598e-05 & 1.429732e-02 & 1.299963e-04 & 4.341353e-01 \\
256 & 5.524272e-03 & 4.0527062295897598e-01 & 1.411161e-05 & 7.117010e-03 & 3.213867e-05 & 4.390850e-01 \\
\bottomrule
\end{tabular}
\end{adjustbox}
\endgroup
\end{table}

\begin{table}[htb]
\centering
\begingroup
\footnotesize
\setlength{\tabcolsep}{3pt}
\sisetup{
  scientific-notation = true,
  round-mode = figures,
  round-precision = 3
}
\caption{Computed enclosure for the quadratic quantity of interest $\mathcal J_2(u)=\tfrac12\int_\Omega u^2\,dx$ on uniform meshes. The columns have the same meaning as in Table~\ref{tab:sec9-qoi-linear}.}
\label{tab:sec9-qoi-quadratic}
\begin{adjustbox}{max width=\linewidth}
\begin{tabular}{r S S S S S S}
\toprule
{$N$} & {$h_{\max}$} & {$\mathcal J_2(\tilde u_h)$} & {$|\mathcal J_2(u^*)-\mathcal J_2(\tilde u_h)|$} &
{$W_{\rm base}$} & {$W_{\rm adj}$} & {err/adj} \\
\midrule
16 & 8.838835e-02 & 1.2280099805235657e-01 & 2.199002e-03 & 8.201916e-02 & 1.250646e-02 & 1.758293e-01 \\
32 & 4.419417e-02 & 1.2444618694850887e-01 & 5.538131e-04 & 3.443567e-02 & 2.879958e-03 & 1.922990e-01 \\
64 & 2.209709e-02 & 1.2486130126557125e-01 & 1.386987e-04 & 1.566278e-02 & 6.894996e-04 & 2.011585e-01 \\
128 & 1.104854e-02 & 1.2496531011195330e-01 & 3.468989e-05 & 7.454290e-03 & 1.685939e-04 & 2.057600e-01 \\
256 & 5.524272e-03 & 1.2499132657988876e-01 & 8.673420e-06 & 3.634359e-03 & 4.167795e-05 & 2.081057e-01 \\
\bottomrule
\end{tabular}
\end{adjustbox}
\endgroup
\end{table}

\subsection{Behavior on anisotropic right-triangular mesh families} \label{subsec:anisotropic-numerics}
We finally examine the behavior of the computed quantities entering the certification tests on anisotropic right-triangular mesh families. This test is motivated by the anisotropic mesh construction used in \cite{IshizakaKobayashiTsuchiya2023,IshizakaPhD2022}, here specialized to two dimensions. The purpose is not to prove a general anisotropic efficiency theorem for the equilibrated estimator, nor to provide a fully interval-verified computer-assisted proof in the sense of Remark~\ref{rem:sec9-floating-point}. The test complements Proposition~\ref{prop:anisotropic-mismatch-consistency} by showing how the flux-mismatch contribution, the residual bound, and the computed verification radius behave under anisotropic refinement.

We use the same monotone semilinear model problem as above,
\begin{align*}
\displaystyle
-\varDelta u + u^3 = f \quad\text{in }\Omega=(0,1)^2, \quad u=0\quad\text{on }\partial\Omega,
\end{align*}
with the manufactured exact solution
\begin{align*}
\displaystyle
u(x,y)=\sin(\pi x)\sin(\pi y).
\end{align*}
The right-hand side is therefore chosen as
\begin{align*}
\displaystyle
f=-\Delta u+u^3.
\end{align*}
The nonlinearity is monotone, so that the coercivity-based verification of Section~\ref{sec:stability} applies.

The anisotropic meshes are generated as follows. Let $M$ be the number of subdivisions in the $x$-direction and let $N$ be the number of subdivisions in the $y$-direction. Following the anisotropic scaling used in \cite{IshizakaKobayashiTsuchiya2023,IshizakaPhD2022}, but in two space dimensions, we set
\begin{align*}
\displaystyle
N=\lfloor M^\gamma\rfloor,\quad \gamma=2.0.
\end{align*}
Thus,
\begin{align*}
\displaystyle
h_x=\frac1M,\quad h_y=\frac1N\simeq M^{-2}, \quad h_y\ll h_x .
\end{align*}
Each rectangle of the tensor-product grid is subdivided into two right triangles. In the two corner rectangles, the diagonal is flipped to avoid a triangle whose three vertices all lie on the boundary. We measure the anisotropy by
\begin{align*}
\displaystyle
\kappa_h:=\frac{h_x}{h_y}=\frac{N}{M}\simeq M^{\gamma-1}=M .
\end{align*}
Therefore, $\kappa_h\to\infty$ as $M\to\infty$, and the mesh family lies outside the usual shape-regular regime. All elements are right triangles, so the maximum angle remains $\pi/2$ even as the aspect ratio increases; this numerical family is anisotropic but does not violate the maximum-angle condition.

For each mesh, we compute the conforming state $\tilde u_h$, the auxiliary CR solution corresponding to the cellwise projection $r_h=\Pi_h^0 s_h$, the reconstructed $\mathbb{RT}^0$ certificate flux $\sigma_h$, and the corresponding scalar quantities entering the verification inequalities. To make the comparison with the uniform-mesh experiments transparent, the anisotropic results are given in the same format as Tables~\ref{tab:sec9-cert}--\ref{tab:sec9-qoi-quadratic}, with $M=N_x$, $N_y=M^2$, $\kappa_h=N_y/M$, and $h_{\max}$ describing the anisotropic mesh. All values are floating-point evaluations of the derived rigorous estimators, as explained in Remark~\ref{rem:sec9-floating-point}.

\begin{figure}[!htbp]
\centering
\includegraphics[width=0.45\textwidth]{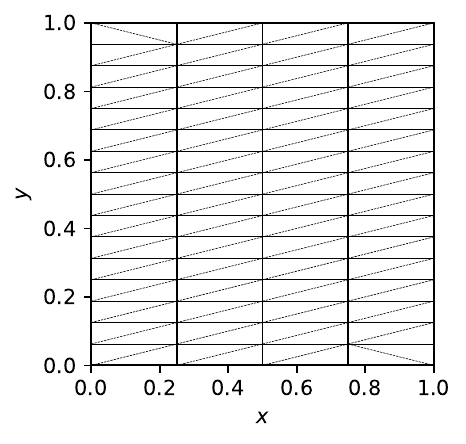}
\caption{Anisotropic right-triangular mesh with $M=4$ and $N=M^2=16$.}
\label{fig:anisotropic-mesh}
\end{figure}

\begin{table}[!htbp]
\centering
\begingroup
\footnotesize
\setlength{\tabcolsep}{1.5pt}
\sisetup{
  scientific-notation = true,
  round-mode = figures,
  round-precision = 3
}

\caption{Computed NK data on anisotropic right-triangular meshes. The columns have the same meaning as in Table~\ref{tab:sec9-cert}, with $M=N_x$, $N_y=M^2$, and $\kappa_h=N_y/M$ indicating the anisotropic mesh parameters. In all rows, the selected radius satisfies $\rho/\eta=2$.}
\label{tab:sec9-aniso-cert}
\begin{adjustbox}{max width=\linewidth}
\begin{tabular}{r r S S S S S S}
\toprule
{$M$} & {$N_y$} & {$\kappa_h$} & {$h_{\max}$} & {$\eta$} & {$\rho$} & {$q(\rho)$} & {$p(\rho)$} \\
\midrule
4 & 16 & 4.000000e+00 & 2.576941e-01 & 8.6813882654070107e-01 & 1.7362776530814021e+00 & 2.6761954692622603e-01 & -6.3580790711281310e-01 \\
8 & 64 & 8.000000e+00 & 1.259728e-01 & 3.9363674398509108e-01 & 7.8727348797018215e-01 & 9.3511222219732945e-02 & -3.5682729094444909e-01 \\
16 & 256 & 1.600000e+01 & 6.262195e-02 & 1.8727645441347479e-01 & 3.7455290882694958e-01 & 3.8554745996108405e-02 & -1.8005605828251151e-01 \\
32 & 1024 & 3.200000e+01 & 3.126526e-02 & 9.1319626661376496e-02 & 1.8263925332275299e-01 & 1.7432481229515791e-02 & -8.9727698983715659e-02 \\
\bottomrule
\end{tabular}
\end{adjustbox}
\endgroup
\end{table}

\begin{table}[!htbp]
\centering
\begingroup
\footnotesize
\setlength{\tabcolsep}{1.5pt}
\sisetup{
  scientific-notation = true,
  round-mode = figures,
  round-precision = 3
}

\caption{Sanity check against the manufactured solution on anisotropic meshes. The columns have the same meaning as in Table~\ref{tab:sec9-sanity}. In particular, ``inside?'' indicates whether $\|u^*-\tilde u_h\|_V\le \rho$.}
\label{tab:sec9-aniso-sanity}
\begin{adjustbox}{max width=\linewidth}
\begin{tabular}{r r S S S S S c}
\toprule
{$M$} & {$N_y$} & {$\kappa_h$} & {$h_{\max}$} & {$\|u^*-\tilde u_h\|_V$} & {$\rho$} & {$\|u^*-\tilde u_h\|_V/\rho$} & {inside?} \\
\midrule
4 & 16 & 4.000000e+00 & 2.576941e-01 & 6.2457116801838131e-01 & 1.7362776530814021e+00 & 3.597185e-01 & yes \\
8 & 64 & 8.000000e+00 & 1.259728e-01 & 3.0966484272332140e-01 & 7.8727348797018215e-01 & 3.933383e-01 & yes \\
16 & 256 & 1.600000e+01 & 6.262195e-02 & 1.5437379770243276e-01 & 3.7455290882694958e-01 & 4.121548e-01 & yes \\
32 & 1024 & 3.200000e+01 & 3.126526e-02 & 7.7126624026315344e-02 & 1.8263925332275299e-01 & 4.222894e-01 & yes \\
\bottomrule
\end{tabular}
\end{adjustbox}
\endgroup
\end{table}

\begin{table}[htb]
\centering
\begingroup
\footnotesize
\setlength{\tabcolsep}{1.5pt}
\sisetup{
  scientific-notation = true,
  round-mode = figures,
  round-precision = 3
}
\caption{Computed enclosure for the linear quantity of interest
$\mathcal J_1(u)=\int_\Omega u\,dx$ on anisotropic meshes. The mesh
parameters are the same as in Table~\ref{tab:sec9-aniso-cert}. The columns give
the computed value, the true error for the manufactured solution, the
baseline and adjoint half-widths, and the ratio err/adj.}
\label{tab:sec9-aniso-qoi-linear}
\begin{adjustbox}{max width=\linewidth}
\begin{tabular}{r r S S S S S}
\toprule
{$M$} & {$N_y$}
& {$\mathcal J_1(\tilde u_h)$}
& {$|\mathcal J_1(u^*)-\mathcal J_1(\tilde u_h)|$}
& {$W_{\rm base}$}
& {$W_{\rm adj}$}
& {err/adj} \\
\midrule
4 & 16 & 3.7536927015309779e-01 & 2.991546e-02 & 3.907998e-01 & 1.146052e-01 & 2.610307e-01 \\
8 & 64 & 3.9795331465985040e-01 & 7.331420e-03 & 1.771988e-01 & 2.134215e-02 & 3.435184e-01 \\
16 & 256 & 4.0347102969420040e-01 & 1.813705e-03 & 8.430402e-02 & 4.641513e-03 & 3.907573e-01 \\
32 & 1024 & 4.0483266748453356e-01 & 4.520671e-04 & 4.110827e-02 & 1.085326e-03 & 4.165264e-01 \\
\bottomrule
\end{tabular}
\end{adjustbox}
\endgroup
\end{table}

\begin{table}[htb]
\centering
\begingroup
\footnotesize
\setlength{\tabcolsep}{1.5pt}
\sisetup{
  scientific-notation = true,
  round-mode = figures,
  round-precision = 3
}
\caption{Computed enclosure for the quadratic quantity of interest
$\mathcal J_2(u)=\frac12\int_\Omega u^2\,dx$ on anisotropic meshes. The
mesh parameters are the same as in Table~\ref{tab:sec9-aniso-cert}. The columns give
the computed value, the true error for the manufactured solution,
the baseline and adjoint half-widths, and the ratio err/adj.}
\label{tab:sec9-aniso-qoi-quadratic}
\begin{adjustbox}{max width=\linewidth}
\begin{tabular}{r r S S S S S}
\toprule
{$M$} & {$N_y$}
& {$\mathcal J_2(\tilde u_h)$}
& {$|\mathcal J_2(u^*)-\mathcal J_2(\tilde u_h)|$}
& {$W_{\rm base}$}
& {$W_{\rm adj}$}
& {err/adj} \\
\midrule
4 & 16 & 1.0734925251575975e-01 & 1.765075e-02 & 4.101660e-01 & 1.401003e-01 & 1.259865e-01 \\
8 & 64 & 1.2055472917759819e-01 & 4.445271e-03 & 1.341089e-01 & 2.722998e-02 & 1.632491e-01 \\
16 & 256 & 1.2388934108467579e-01 & 1.110659e-03 & 5.262508e-02 & 5.994010e-03 & 1.852948e-01 \\
32 & 1024 & 1.2472241185204977e-01 & 2.775881e-04 & 2.306614e-02 & 1.405925e-03 & 1.974416e-01 \\
\bottomrule
\end{tabular}
\end{adjustbox}
\endgroup
\end{table}

{
For this anisotropic family, $h_x=M^{-1}$, $h_y=M^{-2}$, and the maximum element diameter is $h=\sqrt{h_x^2+h_y^2}=M^{-1}\sqrt{1+M^{-2}}\simeq h_x$. For the subsequence $M=4,8,16,32$, the scaled residual parameter $\eta$ and the selected radius $\rho$ decrease steadily when $M$ is doubled. This is consistent with first-order behavior with respect to the maximum mesh size. As a diagnostic check using the manufactured exact solution, the ratios $\|u^*-\tilde u_h\|_V/\rho$ remain below one for all tested meshes, ranging approximately from $0.36$ to $0.42$. Thus, the manufactured solution remains inside the corresponding ball $B(\tilde u_h,\rho)$. In contrast, the adjoint-based output half-widths decrease by factors from about $5.1$--$5.4$ on the first refinement to about $4.3$ on the last, approaching four under refinement and indicating an approximately second-order behavior for the present smooth manufactured solution. Since $\rho=2\eta$ in these tests, the residual effectivity index $\eta/\|u^*-\tilde u_h\|_V$ decreases from about $1.39$ to $1.18$ as $\kappa_h$ grows from $4$ to $32$; on the uniform sequence it decreases from about $1.24$ to $1.16$. This is not a proof of uniform anisotropic efficiency, but it shows that the low-order residual upper bound remains comparably sharp in the reported tests. These observations are consistent with the conditional statement of Proposition~\ref{prop:anisotropic-mismatch-consistency}: once the underlying CR and conforming approximations are energy-consistent on the chosen anisotropic mesh family, the CR--RT reconstruction does not create an additional obstruction in the flux-mismatch term.
}

This numerical test is not intended as a proof of uniform anisotropic efficiency. Its role is more modest: it provides numerical evidence that the low-order CR--RT certification route behaves coherently on the anisotropic right-triangular meshes considered here.

\FloatBarrier
\section{Conclusion}\label{sec:conclusion}
We have constructed a low-order a posteriori test for a conforming approximation of a semilinear diffusion--reaction problem. An auxiliary CR solve and an explicit $\mathbb{RT}^0$ flux give the residual bound for the NK test in Theorem~\ref{thm:NK}. Once this test succeeds, its radius also enters the output estimate in Theorem~\ref{thm:qoi:adjoint}. The latter combines primal and adjoint residual information with explicit remainders from the nonlinear equation and, where needed, the output functional. The adjoint residual is treated by the same flux construction.

The CR--RT formula, NK argument and adjoint representation are known separately. Here they are put to work in one computable setting, with a conditional flux-mismatch consistency result for anisotropic meshes. We have not established uniform efficiency for arbitrary anisotropic meshes or a runtime advantage over other equilibrated methods. The numerical results use double precision and are not interval-verified enclosures.

Several extensions are naturally conceivable. The present analysis is confined to scalar semilinear diffusion--reaction problems. Extensions to Stokes-type incompressible flow models would require a separate treatment of the saddle-point structure, the pressure variable, and pressure-robust velocity certification. Here, pressure-robustness means that velocity bounds should not be polluted by pressure errors or by irrotational forcing components. Although $H(\operatorname{div})$-conforming reconstructions are often used in pressure-robust discretizations and a posteriori analysis, we leave the development of such a certified framework to future work. For time-dependent problems, stepwise certification must be complemented by a rigorous propagation of the verified enclosure in time and by goal-oriented space--time error control to obtain guaranteed bounds for dynamic outputs.

Appendix~\ref{app:rho-selection} compiles deferred details for the practical selection of the verification radius.

\appendix
\section{Proofs of some lemmata}

\subsection{Proof of Lemma \ref{lem:remainder}} \label{Appendix=A=1}

\begin{proof}
Let $w\in B_\rho$ and set $\delta:=w-\tilde u_h$. From the fundamental theorem of calculus in Banach spaces,
\begin{align*}
\displaystyle
\mathcal F(\tilde u_h+\delta) - \mathcal F(\tilde u_h) = \int_0^1 \mathcal L_{\tilde u_h + t \delta} \delta dt.
\end{align*}
which leads to 
\begin{align*}
\displaystyle
\mathcal F(w)-\mathcal F(\tilde u_h)-\mathcal L_{\tilde u_h}\delta
=\int_0^1 \left (\mathcal L_{\tilde u_h + t \delta} - \mathcal L_{\tilde u_h} \right) \delta dt.
\end{align*}
Because $\tilde u_h+t\delta\in B_\rho$ for all $t\in[0,1]$, \eqref{eq:C3} implies
\begin{align*}
\displaystyle
\| \mathcal L_{\tilde u_h + t \delta} - \mathcal L_{\tilde u_h} \|_{\mathcal L(V,V^*)}
\leq L(\rho) t \|\delta\|_{V}.
\end{align*}
Therefore, using the triangle inequality for Bochner integrals,
\begin{align*}
\displaystyle
\|\mathcal F(w)- \mathcal F(\tilde u_h)-\mathcal L_{\tilde u_h}\delta\|_{V^*}
&\leq \int_0^1 \| \left ( \mathcal L_{\tilde u_h + t \delta} - \mathcal L_{\tilde u_h} \right) \delta \|_{V^*} dt \\
&\leq \int_0^1 L(\rho) t \| \delta \|_{V}^2 dt = \frac{L(\rho)}{2}\,\|\delta\|_{V}^{2},
\end{align*}
which proves \eqref{eq:remainder}.
\end{proof}

\subsection{Proof of Lemma \ref{lem:wCS}} \label{Appendix=A=2}

\begin{proof}
Fix $u,v\in L^2(\Omega)^d$.
For a.e.\ $x\in\Omega$, we define the pointwise bilinear form as
\begin{align*}
\displaystyle
\langle \xi,\eta\rangle_{A(x)} := \xi\cdot A(x)\eta,
\quad \xi,\eta\in\mathbb R^d.
\end{align*}
Because $A(x)$ is symmetric positive definite, $\langle\cdot,\cdot\rangle_{A(x)}$ is an inner product on $\mathbb R^d$. Hence, the (finite-dimensional) Cauchy--Schwarz inequality yields, for a.e.\ $x\in\Omega$,
\begin{align*} 
\displaystyle
|u(x)\cdot v(x)|
&=
\left | \langle A(x)^{-1} u(x), v(x) \rangle_{A(x)} \right| \notag \\
&\leq
\left ( u(x) \cdot A(x)^{-1}u(x) \right)^{\frac{1}{2}}
\left ( v(x) \cdot A(x) v(x) \right )^{\frac{1}{2}}.
\end{align*}
Integrating the above inequality over $\Omega$ and applying the standard Cauchy--Schwarz inequality in $L^2(\Omega)$ gives
\begin{align*}
\displaystyle
\left |\int_\Omega u\cdot v\,dx\right|
\leq
\int_\Omega
\left ( u\cdot A^{-1}u \right)^{\frac{1}{2}}
\left ( v\cdot A v \right)^{\frac{1}{2}}\,dx
\leq
\left (\int_\Omega u\cdot A^{-1}u dx\right)^{\frac{1}{2}}
\left (\int_\Omega v\cdot A v dx \right)^{\frac{1}{2}},
\end{align*}
which is \eqref{eq:wCS}.
\end{proof}

\subsection{Proof of Lemma \ref{thm:res-majorant}} \label{Appendix=A=3}

\begin{proof}
Fix any $v\in V$. From the residual representation \eqref{eq:res_split},
\begin{align*}
\langle \mathcal F(\tilde u_h),v\rangle
= \int_\Omega\left ( A \nabla \tilde u_h - \sigma_h \right)\cdot\nabla v dx
+\sum_{T\in\mathbb T_h}\int_T \mathcal R_T(\sigma_h) v dx. 
\end{align*}
Let $w:=\sigma_h - A \nabla \tilde u_h \in L^2(\Omega)^d$. Then,
\begin{align*}
\displaystyle
\left |\int_\Omega\left( A \nabla \tilde u_h - \sigma_h \right)\cdot\nabla v dx\right|
= \left |\int_\Omega w\cdot\nabla v dx\right|.
\end{align*}
Using the weighted Cauchy--Schwarz inequality \eqref{eq:wCS} associated with $A$,
\begin{align*}
\displaystyle
\left |\int_\Omega w\cdot\nabla v\,dx \right|
\leq \left (\int_\Omega w\cdot A^{-1}w dx \right)^{\frac{1}{2}}
     \left (\int_\Omega A \nabla v\cdot\nabla v dx\right)^{\frac{1}{2}}
= \eta_{\mathrm{mis}}(\sigma_h) \|v\|_V.
\end{align*}
For each $T \in \mathbb T_h$, we denote the element average  $\bar v_T := |T|_d^{-1}\int_T v dx$. From the equilibration \eqref{eq:mean-equil},
\begin{align*}
\displaystyle
\int_T \mathcal R_T(\sigma_h) v dx
=\int_T \mathcal R_T(\sigma_h) (v-\bar v_T) dx.
\end{align*}
The Cauchy--Schwarz and the Poincar\'e inequalities \eqref{eq:local-poincare} yield
\begin{align*}
\displaystyle
\left| \int_T \mathcal R_T(\sigma_h) v dx \right|
\leq \| \mathcal R_T(\sigma_h)\|_{L^2(T)} \|v-\bar v_T\|_{L^2(T)}
\leq \|\mathcal R_T(\sigma_h)\|_{L^2(T)} \frac{h_T}{\pi}\,\|\nabla v\|_{L^2(T)}.
\end{align*}
Summing over $T$ and applying Cauchy--Schwarz yields
\begin{align*}
\displaystyle
\sum_{T\in\mathbb T_h}\left |\int_T \mathcal R_T(\sigma_h) v dx\right|
\leq \eta_{\mathrm{osc}}(\sigma_h) \|\nabla v\|_{L^2(\Omega)^d}.
\end{align*}
\eqref{eq:A0-coerc} implies
\begin{align*}
\displaystyle
\|v\|_V^2
=\int_\Omega A \nabla v\cdot\nabla v dx
\geq \alpha_0\int_\Omega |\nabla v|^2 dx
=\alpha_0 \|\nabla v\|_{L^2(\Omega)^d}^2,
\end{align*}
which leads to
\begin{align*}
\displaystyle
\|\nabla v\|_{L^2(\Omega)^d}\leq \alpha_0^{-1/2}\|v\|_V.
\end{align*}
Therefore,
\begin{align*}
\displaystyle
\sum_{T\in\mathbb T_h}\left |\int_T \mathcal R_T(\sigma_h) v dx \right|
\leq \alpha_0^{- \frac{1}{2}}\,\eta_{\mathrm{osc}}(\sigma_h) \|v\|_V.
\end{align*}
Combining the above results gives, for all $v\in V$,
\begin{align*}
\displaystyle
|\langle \mathcal F(\tilde u_h),v\rangle|
\leq \left (\eta_{\mathrm{mis}}(\sigma_h)+\alpha_0^{- \frac{1}{2}}\eta_{\mathrm{osc}}(\sigma_h)\right) \|v\|_V.
\end{align*}
Dividing by $\|v\|_V$ and taking the supremum over $v\neq 0$ yields \eqref{eq:dual_res_bound}.
\end{proof}

\section{Deferred material for the radius selection}\label{app:rho-selection}
This appendix provides the technical justification of the certified radius-selection procedure (Algorithm~\ref{alg:rho}).

\subsection{Algorithm and basic properties}\label{app:rho:algo}
The following radius-selection procedure is stated for the affine bound $L(\rho)=L_0+L_1\rho$ with $L_0\geq0$ and $L_1>0$, on the permitted range $0<\rho\leq\bar\rho$. It uses the minimum of $p$ to find an admissible radius, if one exists, and optionally applies bisection to approach its upper endpoint. The first test, $\rho=2\eta$, is retained because it is the radius used in the numerical experiments.

\begin{definition}[Admissible radius]\label{def:admissible-radius}
For $\rho>0$, we define the admissibility predicate
\begin{align*}
\displaystyle
\mathsf{Adm}(\rho) \Longleftrightarrow  \left (p(\rho) \leq 0 \right) \ \text{and}\ \left (q(\rho)<1 \right),
\end{align*}
and the set of admissible radii
\begin{align*}
\displaystyle
\mathcal{AD} := \{0<\rho\leq\bar\rho:\ \mathsf{Adm}(\rho)\}.
\end{align*}
\end{definition}

\begingroup
\begin{algorithm}[htb]
\caption{Selection of the verification radius for $L(\rho)=L_0+L_1\rho$}\label{alg:rho}
\begin{algorithmic}[1]
\State Input certified $\alpha>0$, $\mathfrak r\geq0$, $\bar\rho>0$, $L_0\geq0$, $L_1>0$, and integer $k_{\max}\geq0$.
\State Set $\eta\gets\mathfrak r/\alpha$ and $\rho_0\gets 2\eta$.
\If{$0<\rho_0\leq\bar\rho$ \textbf{and} $\mathsf{Adm}(\rho_0)$}
  \State \Return $\rho_0$
\EndIf
\State $\rho_m\gets 2\alpha/(L_0+\sqrt{L_0^2+6\alpha L_1})$
\State $\rho_{\rm low}\gets\min\{\rho_m,\bar\rho\}$
\If{\textbf{not} $\mathsf{Adm}(\rho_{\rm low})$}
  \State \Return \texttt{fail} \Comment{no admissible radius in $(0,\bar\rho]$}
\EndIf
\If{$\rho_{\rm low}=\bar\rho$}
  \State \Return $\rho_{\rm low}$
\EndIf
\State $\rho_{\rm up}\gets\bar\rho$
\If{$\mathsf{Adm}(\rho_{\rm up})$}
  \State \Return $\rho_{\rm up}$
\EndIf
\State \Comment{Bisection on the upper boundary of the admissible interval}
\For{$k=1,\dots,k_{\max}$}
  \State $\rho_{\rm mid}\gets(\rho_{\rm low}+\rho_{\rm up})/2$
  \If{$\mathsf{Adm}(\rho_{\rm mid})$}
    \State $\rho_{\rm low}\gets\rho_{\rm mid}$
  \Else
    \State $\rho_{\rm up}\gets\rho_{\rm mid}$
  \EndIf
\EndFor
\State \Return $\rho_{\rm low}$
\end{algorithmic}
\end{algorithm}
\endgroup

\subsection{Monotonicity and convexity for the affine model}\label{app:rho:proofs}
We provide here the proofs of the monotonicity and convexity statements used implicitly in Section~\ref{subsec:rho-selection}. For an affine Lipschitz bound, the admissible radii form an interval. The minimizer of $p$ provides an admissible starting point whenever this interval is nonempty. Bisection is then valid on the upper side of that interval.

\begingroup
\begin{lemma}[Interval structure of admissible radii for an affine Lipschitz bound]\label{lem:adm-interval}
Let $\bar\rho>0$, $\alpha>0$ and $\mathfrak r\geq0$ be given, and set
\begin{align*}
\displaystyle
\eta := \frac{\mathfrak r}{\alpha}.
\end{align*}
Assume that the Lipschitz bound has the affine form
\begin{align*}
\displaystyle
L(\rho)=L_0+L_1\rho, \quad  L_0 \geq 0,\quad L_1 >0,
\end{align*}
which includes the standard choice $L(\rho)=6C_4^4(\|\tilde u_h\|_V+\rho)$.
For $\rho>0$, define the admissibility predicate by
\begin{align*}
\displaystyle
\mathsf{Adm}(\rho)
\Longleftrightarrow
\left(p(\rho)\leq 0\right)\ \text{and}\ \left(q(\rho)<1\right),
\end{align*}
where
\begin{align*}
\displaystyle
p(\rho) = \eta+\frac{1}{2}\alpha^{-1}L(\rho)\rho^2-\rho,
\quad
q(\rho)=\alpha^{-1}L(\rho)\rho.
\end{align*}
Let
\begin{align*}
\displaystyle
\mathcal{AD}:=\{0<\rho\leq\bar\rho:\mathsf{Adm}(\rho)\}.
\end{align*}
Then, $\mathcal{AD}$ is empty, a singleton, or an interval. When $\eta>0$ and $\mathcal{AD}$ contains more than one point, it has the form $[\rho_-,\rho_+]$ or $[\rho_-,\rho_+)$ with $0<\rho_-<\rho_+\leq\bar\rho$. If $\eta=0$, the admissible interval instead starts at zero, which is excluded. If $\mathsf{Adm}(\rho_{\rm low})$ holds and $\mathsf{Adm}(\rho_{\rm up})$ fails for $0<\rho_{\rm low}<\rho_{\rm up}\leq\bar\rho$, then $\rho_{\rm low}\leq\sup\mathcal{AD}\leq\rho_{\rm up}$, and bisection with admissible lower endpoints preserves this bracket.
\end{lemma}

\begin{proof}
First, consider the contraction indicator. Since
\begin{align*}
\displaystyle
q(\rho)=\alpha^{-1}L(\rho)\rho
=\alpha^{-1}(L_0\rho+L_1\rho^2),
\end{align*}
we have
\begin{align*}
\displaystyle
q'(\rho)=\alpha^{-1}(L_0+2L_1\rho)>0,\quad \rho>0.
\end{align*}
Thus $q$ is strictly increasing on $(0,\infty)$. Moreover, since $L_1>0$, one has $q(\rho)\to\infty$ as $\rho\to\infty$. Therefore,
\begin{align*}
\displaystyle
\{\rho>0:q(\rho)<1\}=(0,\rho_q)
\end{align*}
for a unique $\rho_q\in(0,\infty)$.

Next, expand $p$ as
\begin{align*}
\displaystyle
p(\rho)=\eta+\frac{1}{2}\alpha^{-1}(L_0\rho^2+L_1\rho^3)-\rho.
\end{align*}
Then
\begin{align*}
\displaystyle
p''(\rho)=\alpha^{-1}(L_0+3L_1\rho)>0,\quad \rho>0.
\end{align*}
Hence $p$ is strictly convex on $(0,\infty)$. Extending $p$ continuously to $\rho=0$, we have
$p(0)=\eta\geq0$. Also, $p(\rho)\to\infty$ as $\rho\to\infty$.
Therefore, the sublevel set
\begin{align*}
\displaystyle
\{\rho>0:p(\rho)\leq 0\}
\end{align*}
is either empty, a singleton, or a compact interval
$[\rho_p^-,\rho_p^+]$ with $0<\rho_p^-\leq \rho_p^+<\infty$ when $\eta>0$. When $\eta=0$, the same derivative argument gives $(0,\rho_p^+]$ for some $\rho_p^+>0$.

By definition,
\begin{align*}
\mathcal{AD}=\{\rho>0:p(\rho)\leq 0\}\cap\{\rho>0:q(\rho)<1\}\cap(0,\bar\rho].
\end{align*}
The intersection of the above sublevel set, the interval defined by $q(\rho)<1$, and $(0,\bar\rho]$ has the stated structure. If $\rho_{\rm low}$ is admissible and a larger $\rho_{\rm up}$ is not, the upper endpoint of the admissible interval lies between them. Each bisection step keeps the lower endpoint admissible and the upper endpoint inadmissible. The endpoint need not itself be admissible because $q(\rho)<1$ is strict.
\end{proof}

\begin{proposition}[Correctness of Algorithm~\ref{alg:rho}]\label{prop:alg-rho-correct}
Let $\bar\rho>0$, $\alpha>0$, $\mathfrak r\geq0$ and $L(\rho)=L_0+L_1\rho$ with $L_0\geq0$, $L_1>0$. In exact arithmetic, Algorithm~\ref{alg:rho} returns an admissible radius whenever $\mathcal{AD}$ is nonempty, and returns \texttt{fail} only if $\mathcal{AD}$ is empty. If the algorithm enters bisection, the returned $\rho_{\rm low}$ is admissible and, with $\rho_+:=\sup\mathcal{AD}$, satisfies
\begin{align*}
\rho_{\rm low}\leq\rho_+\leq\rho_{\rm up},\quad
\rho_{\rm up}-\rho_{\rm low}\leq 2^{-k_{\max}}\bigl(\rho_{\rm up}^{(0)}-\rho_{\rm low}^{(0)}\bigr).
\end{align*}
\end{proposition}

\begin{proof}
\begingroup
By differentiation,
\begin{align*}
p'(\rho)=\alpha^{-1}\bigl(L_0\rho+\tfrac32 L_1\rho^2\bigr)-1,
\quad
p''(\rho)=\alpha^{-1}(L_0+3L_1\rho)>0.
\end{align*}
Thus, $p$ has the unique minimizer $\rho_m>0$ given by
\begin{align*}
\rho_m=\frac{2\alpha}{L_0+\sqrt{L_0^2+6\alpha L_1}}.
\end{align*}
At this point,
\begin{align*}
q(\rho_m)=\alpha^{-1}(L_0\rho_m+L_1\rho_m^2)
=1-\frac{L_1\rho_m^2}{2\alpha}<1.
\end{align*}
Because $q$ is increasing, $q(\rho)<1$ for all $0<\rho\leq\rho_m$. Consequently, $\rho_{\rm seed}:=\min\{\rho_m,\bar\rho\}$ is admissible precisely when $\mathcal{AD}$ is nonempty: if $p(\rho_{\rm seed})>0$, the minimum of $p$ over the permitted interval is positive; otherwise $p(\rho_{\rm seed})\leq0$ and $q(\rho_{\rm seed})<1$. Thus the algorithm cannot miss an admissible interval during its initial search. The preliminary $2\eta$ test, if successful, is already admissible.

If $\rho_{\rm seed}=\bar\rho$, or if $\bar\rho$ is admissible, the returned radius is admissible. In the remaining case $\rho_{\rm seed}<\bar\rho$, the pair $(\rho_{\rm seed},\bar\rho)$ brackets the upper endpoint of $\mathcal{AD}$ by Lemma~\ref{lem:adm-interval}. Every bisection step replaces exactly one endpoint by its midpoint, preserving the admissibility of the lower endpoint and the failure of admissibility of the upper endpoint. The claimed interval-width estimate follows after $k_{\max}$ steps.
\endgroup
\end{proof}

\endgroup

\begin{remark}[Sanity check against the manufactured solution]
Because $u^*$ is known here, we can compare the certified radius with the true error:
\begin{align*}
\displaystyle
\|u^*-\tilde u_h\|_V \leq \rho \quad \text{(expected to hold if $\tilde u_h$ is close to $u^*$)},
\end{align*}
see Table \ref{tab:sec9-sanity}. We include this comparison only as a diagnostic; certification of existence does not rely on knowing $u^*$.
\end{remark}



\end{document}